\documentclass[a4paper]{llncs}
\usepackage{a4wide}
\usepackage[utf8]{inputenc}
\usepackage[center]{subfigure}
\usepackage{xspace}
\usepackage{graphicx}
\usepackage{amsmath}
\usepackage{amssymb}
\usepackage{comment}
\usepackage{psfrag}
\def\eg{\emph{e.g.}\xspace}
\def\ie{\emph{i.e.}\xspace}

\usepackage{hyperref}
\usepackage{subfigure}
\usepackage{color}
\usepackage{multirow}
\usepackage{graphicx}
\usepackage{braket}

\newcommand{\bea}{\begin{eqnarray}}
\newcommand{\eea}{\end{eqnarray}}
\newcommand{\beas}{\begin{eqnarray*}}
\newcommand{\eeas}{\end{eqnarray*}}
\newcommand{\be}{\begin{equation}}
\newcommand{\ee}{\end{equation}}
\newcommand{\ba}{\begin{array}}
\newcommand{\ea}{\end{array}}

\newcommand{\cV}{{\cal{V}}}
\newcommand{\cE}{{\cal{E}}}

\title{Regular colored graphs of positive degree}
\author{Razvan Gurau$^{a}$ and Gilles Schaeffer$^{b}$
\thanks{GS acknowledge the support of 
ERC via Research Starting Grant 208471 {ExploreMaps}.}}
\institute{   (a) CPHT, CNRS, UMR 7644, Ecole Polytechnique, rgurau@cpht.polytechnique.fr \\
              (b) LIX, CNRS, UMR 7161, Ecole Polytechnique, schaeffe@lix.polytechnique.fr}
\begin{document}
\maketitle
\begin{abstract}
  Regular colored graphs are dual representations of pure colored
  $D$-dimensional complexes. These graphs can be classified with
  respect to a positive integer, their degree, much like maps are
  characterized by the genus.  We analyze the structure of regular
  colored graphs of fixed degree and perform their exact and
  asymptotic enumeration.  In particular we show that the generating
  function of the family of graphs of fixed degree is an algebraic
  series with a positive radius of convergence, independent of the
  degree. We describe the singular behavior of this series near its
  dominant singularity, and use the results to establish the double
  scaling limit of colored tensor models: interestingly the behavior
   is qualitatively very different for $3\leq D\leq5$ and for $D\geq6$.
\end{abstract} 

\section{Introduction}
\subsubsection{Context.}
In this article a \emph{colored graph} is a rooted connected bipartite graph such that
each edge has a color in $\{0,1,\dots, D\}$ and each vertex is
incident to exactly one edge of each color.  Colored graphs appear
naturally in the crystallization theory of manifolds \cite{Lins} and
in colored tensor models \cite{color} (or colored group field
theory). They are dual to colored triangulations of piecewise linear
orientable $(D+1)$-dimensional pseudo-manifolds \cite{FG,lost}.
Although not all $(D+1)$-triangulations can be properly colored,
colored graphs are fundamental because any orientable topological
manifold in any dimension admits a colored triangulation
\cite{pezzana} and any triangulation in any dimension can be
transformed into a colored triangulation by a barycentric
subdivision.

To each colored graph is associated an invariant, its {\it
  degree} \cite{1overN}, which is a non negative integer. 
  For $D=2$ the degree reduces to the genus of
the dual (2-dimensional) triangulation. Unlike the genus however, the
degree is not a topological invariant of the dual
pseudo-manifold for $D\geq3$.  Be that as it may, classifying graphs
in terms of the degree offers a first rough classification of
triangulations of pseudo-manifolds in any dimension. It also plays a
distinctive role in tensor models, where this classification allows
access to subsequent orders in their $1/N$ expansion, as this
expansion is indexed by the degree (exactly like the $1/N$ expansion
of matrix models is indexed by the genus).

\subsubsection{Our results.}

Our main result is a structural analysis of rooted colored graphs of
fixed degree, which yields on the one hand an exact and an
asymptotic enumeration of these graphs, and on the other hand leads to the
construction of the double scaling limit of colored tensor models. 
  
The structural analysis we perform relies on the reduction of colored
graphs via a precise algorithm to some terminal forms of the same
degree, which we call \emph{reduced schemes}. The number of reduced schemes
of a given degree is finite and the number of graphs sharing a scheme
is exponentially bounded. More precisely we show: 
\begin{theorem}\label{thm:main1}
For any fixed dimension $D\geq3$ and degree $\delta\geq 0$, there exist a
finite set $\tilde {\cal S}^0_{\delta}$ of reduced schemes of degree $\delta$ and root edge of color $0$, 
and triples $( P_{\tilde S}(u), {\bf U}_{\tilde S}, {\bf B}_{\tilde S} )_{ \tilde S\in \tilde {\cal S}^0_{\delta} }$
consisting in a monomial and two integer parameters associated to the schemes such that the generating
function of colored graphs of degree $\delta$ rooted at an edge of color $0$ with respect to the number of black vertices is:
  \[
  H^0_{\delta}(z)=T(z)\sum_{ \tilde S \in \tilde {\cal S}^0_{\delta}}
     \left[ \frac{P_{\tilde S}(u)}{ ( 1-u^2 )^{  {\bf U}_{\tilde S} + {\bf B}_{\tilde S} } (1-D^2u^2)^{{\bf B}_{\tilde S}}}  \right]_{u=zT(z)^{D+1}} \;,
  \]
where $T(z)$ is the unique power series solution of the equation:
   \[ 
   T(z)= 1 + zT(z)^{D+1} \; .
   \]
\end{theorem}

Previous classifications in terms of the degree exist \cite{1overN},
but, while the number of terminal forms identified in \cite{1overN} is
finite at fixed degree, there is no control over the number of graphs
associated to a terminal form.  Our approach is instead reminiscent of
the classification of maps of fixed genus performed in
\cite{chapuy-marcus-schaeffer}, or that of simplicial decompositions
of surfaces with boundaries in \cite{bernardi-rue}, and more generally
of Wright's approach to the enumeration of labeled graph with fixed
excess \cite{wright1,wright2}.

From our main theorem we are able to extract the leading terms in the singular
expansion of the generating functions of colored graphs of degree $\delta$. 
\begin{theorem}\label{thm:main2}
For any fixed $D\geq3$ and $\delta\geq 1$, the generating function of colored
graphs of degree $\delta$ has a dominant singularity at
$z_0=D^D/(D+1)^{D+1}$ and a singular expansion in a slit domain
around $z_0$ of the form:
\[
  H^0_{\delta}(z)=K_\delta(1-z/z_0)^{-\frac{ {\bf B}_{\rm max} }{2}}  \left[ 1+ O \left( \sqrt{1-\frac{z}{z_0}} \right) \right] \;,
\] 
where ${\bf B}_{\rm max}$ is the maximum of a simple integer linear program:
\[
{\bf B}_{\rm max}=\max(2 c_+ +3 q-1\mid (D-2)c_+ +Dq \le \delta ; \;\;  c_+, q \in\mathbb{N}) \; .
\]
In particular ${\bf B}_{\rm max}$ roughly grows linearly with $\delta$ and for
fixed $D$ we determine the largest linearity factor
$\max({\bf B}_{\rm max}/\delta)$ and for which $\delta$ it is obtained:
\[
\begin{array}{c|c|c|c|}
& 3\leq D\leq 5 & D=6 & D\geq 7\\\hline 
  \max({\bf B}_{\rm max}/\delta) &
  \frac{2}{D-2}&\frac{2}{D-2}=\frac{3}{D}
  &\frac{3}{D}\\\hline
  \rm{which} \; \delta & \delta = \mathbb{N} \cdot (D-2) & \textrm{all }\delta & \delta = \mathbb{N} \cdot D  \\\hline
\end{array}
\]
Moreover the constants $K_\delta$ have combinatorial interpretations,
  which for $3\leq D\leq 5$ involve Catalan numbers.
\end{theorem}

\subsubsection{Discussion.}
From a probabilistic point of view the above result implies that we
can give a description of large random colored graphs of fixed degree.
It was shown in \cite{gurauryan} that upon scaling edge lengths
by a factor $k^{-1/2}$ and letting $k$ go to infinity, the degree 0
colored graphs with $2k$ vertices converge in the sense of
Hausdorff-Gromov to the Continuum Random Tree. Our results suggest that 
more generally $k^{-1/2}$ is the proper scaling for which uniform random rooted
colored graphs of fixed degree $\delta \geq 1$ have a non-trivial
continuum limit when the number of vertices goes to infinity.

Another major outcome of our results is the so-called \emph{double
  scaling limit} of colored tensor models. Although the number of
colored graphs with $2k$ vertices grows super-exponentially with $k$,
we can give a meaning to a resummation of the generating series of
graphs of fixed degree. Balancing the singular behavior of these
generating series around the critical point $z_0$ with the scaling in
$N$ we can take the double limit $N\to\infty$, $z\to z_0$ in a
correlated way and exhibit a regime in which graphs with arbitrary large
degree contribute.  As suggested by Theorem~\ref{thm:main2}, this regime leads to two completely diverse behaviors, depending whether $D\leq5$, where the double scaling limit series is summable, or $D\geq 6$ where it is not.
Together with the parallel result obtained in
\cite{ongoing} by different methods and for a simpler model, 
these are the first results of this kind in the realm of tensor models.

A number of very difficult questions remain open. Prominent among them
is the following. A given topology (say spherical) can be represented
by graphs of arbitrary degree. It is a difficult open question whether
the number of triangulations of a fixed topological manifold is
exponentially bounded or not in the number of simplices (the so called Gromov question
\cite{raresph} in the case of the spherical topology). In view of our
results the question can now translate in finding an exponential bound
on the number of reduced schemes to which graphs representing
a given topology can reduce.

\subsubsection{Organization of the paper}

In Section~\ref{sec:notations} we state some definitions and elementary properties of colored
graphs. In Section~\ref{sec:core} we perform a first classification of colored graphs in 
terms of \emph{cores}. In Section~\ref{sec:chains} we discuss \emph{chains} 
and in particular show that the number of cores of fixed degree is not finite due to the presence of infinite chains. This
leads us to the notion of \emph{reduced schemes} in Section~\ref{sec:schemes}, where we show that 
the number of reduced schemes of fixed degree is finite. Sections~\ref{sec:prf1} and~\ref{sec:prf2}
contain the proofs of two technical results. In Section~\ref{sec:enumeration}
we compute the generating series of graphs associated to a reduced scheme 
and in Section~\ref{sec:asymptotics} we identify the reduced schemes with leading singular behavior
at criticality.

\newpage 

\section{Notation and generalities on colored graphs}\label{sec:notations}

\begin{figure}[t]
\begin{center}
\psfrag{G}{$G$}
\psfrag{oG}{$\mathrm{op}(G)$}
\psfrag{r}{$r(G)$}
\includegraphics[scale=.5]{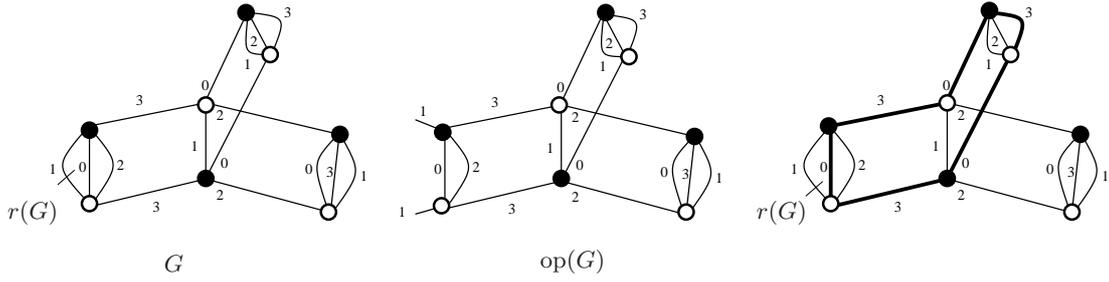}
\end{center}
\caption{A colored graph $G$ (where the root edge is represented as crossed), the open colored graph $\mathrm{op}(G)$, and a face $(0,3)$ of $G$.}\label{fig:colored}
\end{figure}

From now on in this article, an integer $D\geq3$ is fixed.

\begin{definition} A \emph{rooted, closed, connected colored graph} $G$ (henceforth called a \emph{colored graph} for short) is a connected 
bipartite $(D+1)$-regular graph with black and white vertices and colored edges, such that:
\begin{itemize}
 \item the colors of edges are taken in the set $\{0,\ldots,D\}$,
 \item each vertex is incident to exactly one edge of each color,
 \item an edge of $G$, denoted  $r(G)$, is distinguished and it is called the \emph{root edge}.
\end{itemize}

\end{definition}

Multiple edges are allowed, but, due to the color constraints, self-loops are not. An example is presented in Fig.~\ref{fig:colored}
on the left. 

We denote colored graphs by capital letters like $G,G_1$, etc.. 
We distinguish white vertices by a white dot index ($v_{\circ},w_{\circ}$, etc.)
and black vertices by a black dot index ($v_{\bullet}, w_{\bullet}$, etc.).

We include among the colored graphs the \emph{trivial colored graphs} (or \emph{ring graphs})
consisting in an edge closing onto itself and having no vertex (see Fig.~\ref{fig:ringdef} on the left). 
The edge is necessarily the root of the graph and has a color $c\in \{0,\ldots,D\}$, hence there are $D+1$ distinct ring graphs.
We denote them $R^c$, where $c$ is the color of the edge.

\begin{figure}[ht]
\begin{center}
\psfrag{c}{$c$}
\includegraphics[scale=.5]{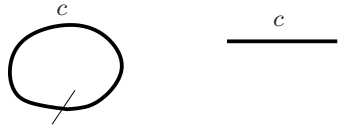}
\end{center}
\caption{A ring graph $R^c$ and a trivial open graph $\mathbb{E}^c=\textrm{op}(R^c) $.}\label{fig:ringdef}
\end{figure}

It will be convenient to regard the edges of a graph as
being formed by pairs of matched half-edges (each half-edge being
hooked to one of the end vertices of the edge) and to allow for
slightly more general structures, called \emph{pre-graphs}, in which
some of the half-edges are left unmatched. From here on half-edges
incident to white vertices will be denoted with a white dot index (\eg
$h_{\circ}, h'_{\circ}$, etc.)  and half-edges incident to black
vertex will be denoted with a black dot index (\eg $h_{\bullet},
h'_{\bullet}$, etc.).

\begin{definition}
A \emph{nontrivial open colored graph} $\mathbb{G}$ is a pre-graph with exactly two unmatched half-edges $h_{\bullet} $ and $h_{\circ}$ of the same color $c$, which 
becomes a colored graph $G$ by matching the two half-edges into an edge of color $c$ and marking this edge as the root edge $r(G)$ of $G$. 

A \emph{trivial open colored graph} $\mathbb{E}^c$ consists in a unique edge of color $c$ with no end vertex
and becomes a ring graph $R^c$ upon matching the two ends of the edge.
\end{definition}

Open colored graphs will be denoted by emphasized capital letters (\eg $\mathbb{G},\mathbb{G}_1$, etc.).

We denote $\textrm{cl}(\mathbb{G})$ (and call it the \emph{closure} of $\mathbb{G}$) the colored graph obtained by matching the two half-edges 
of the open colored graph $\mathbb{G}$ into a root edge. Conversely, given a colored graph $G$, we denote $\textrm{op}(G)$ (and call it the opening of $G$) the unique open colored graph
$\mathbb{G}$ such that $\textrm{cl}(\mathbb{G})=G$ (see  Fig.~\ref{fig:colored} in the center for an example). Of course, $\textrm{cl}(\mathbb{E}^c)=R^c$ and $\textrm{op}(R^c)=\mathbb{E}^c$ as depicted in 
Fig.~\ref{fig:ringdef} on the right.
Open colored graphs are not rooted. 

A non trivial open colored graph $\mathbb{G}$ (which is a pre-graph,
having two half-edges) can be transformed into a graph $
\textrm{Gr}(\mathbb{G})$ by simply erasing the half-edges. Two of the
vertices of $\textrm{Gr}(\mathbb{G})$ have coordination $D$, while all
the others have coordination $D+1$.  The edges of
$\textrm{Gr}(\mathbb{G})$ are colored.

The following definition is illustrated in figure \ref{fig:opensubgraphs}.

\begin{definition}
An \emph{open colored subgraph} $\mathbb{H}$ of a nontrivial open colored graph $\mathbb{G}$, denoted $\mathbb{H} \subset \mathbb{G}$,
is a non trivial open colored graph such that $\mathrm{Gr}(\mathbb{H})$ is a subgraph of $\mathrm{Gr}(\mathbb{G})$.

By extension an \emph{open colored subgraph} $\mathbb{H}$ of a colored graph $G$, denoted $\mathbb{H} \subset G$ is an open colored subgraph of $\mathrm{op}(G)$.
\end{definition}

\begin{figure}[ht]
\begin{center}
\includegraphics[scale=.3]{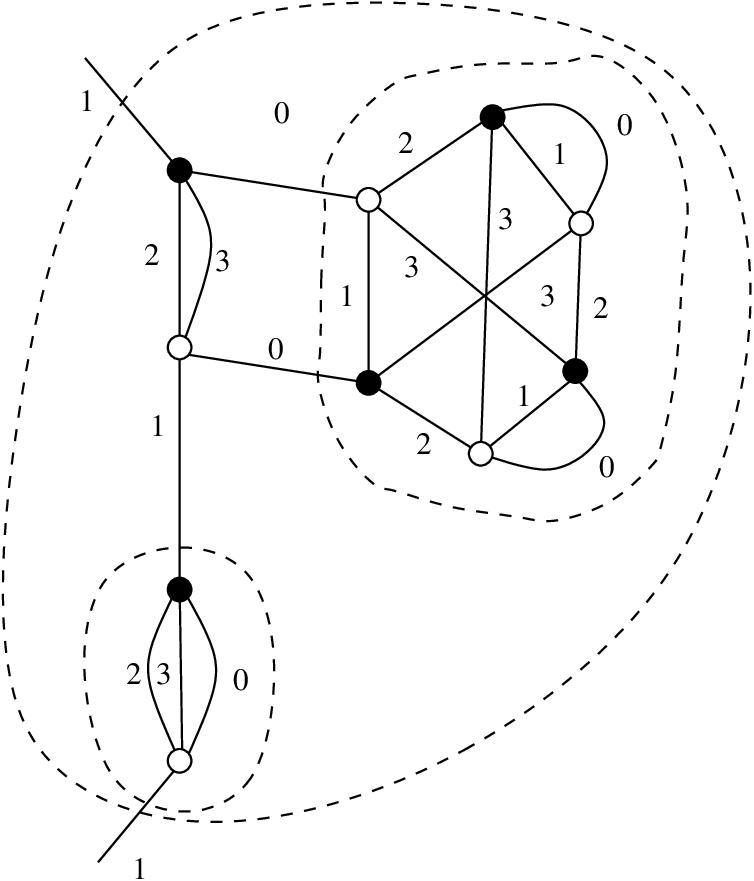}
\end{center}
\caption{Examples of open colored subgraphs of an open colored graph.}\label{fig:opensubgraphs}
\end{figure}

A half-edge, say $h_{\circ}$, of an open colored subgraph $\mathbb{H}\subset \mathbb{G}$ can either 
be one of the half-edges of $\mathbb{G}$ or belong to an edge $e$ of $\mathbb{G}$, which we denote $h_{\circ} \in e$.

Let $G$ be a colored graph. In view of the bipartiteness and
regularity constraints, $G$ has an equal number $k(G)$ of black and white
vertices, and by construction it also has $(D+1)k(G)$ edges.  Let us
define the faces of $G$ as its bicolored connected components: more
precisely, given $0\leq c \neq c'\leq D$, the \emph{faces of color
  $\{c,c'\}$} are the connected components of the subgraph consisting of
edges that have color $c$ or $c'$, and they form a set of cycles
since every vertex in the subgraph has degree 2 (see Fig.~\ref{fig:colored} on the right for an example). 
Observe that with this definition the ring graph $R^c$ has $D$ faces, one for each color different from $c$.

Let $F^{cc'}_p(G)$ denote the number of faces of color $\{c,c'\}$ and length
$2p$ of the colored graph $G$. We denote $\smash{F^{cc'}(G)=\sum_{p\geq1}F_p^{cc'}(G)}$, $\smash{F_p(G)=\sum_{0\le c<c' \le D}F_p^{cc'}(G)}$ 
and $\smash{F(G)=\sum_{p\ge 1}F_p(G)}$, the total number of faces of color $\{c,c'\}$ of $G$, the total number of faces of length $2p$ of $G$ 
and the total number of faces of $G$. 

\begin{definition} Let the \emph{reduced degree} (the \emph{degree} for short) of $G$ be the integer
$\delta(G)$ defined by the relation: 
\begin{equation}\label{eq:defdegree}
\delta(G) = \frac12D(D-1)k(G) + D -F(G) \; .
\end{equation}
\end{definition}

\begin{proposition}[\cite{color}]
The degree $\delta(G)$ of a non trivial colored graph $G$ is a non negative
integer.
\end{proposition}
\begin{proof}
Each of the $D!$ cyclic permutations $\pi$ of the colors
$\{0,\ldots,D\}$ induces a unique embedding of the graph $G$ in a
compact oriented surface upon requiring that $\pi$ describes the
clockwise arrangement of edges around black vertices and the
counterclockwise arrangement of edges around white ones. The resulting
combinatorial map $G_\pi$, called in \cite{color} a \emph{jacket} of
$G$, has $(D+1)k(G)$ edges, $2k(G)$ vertices and $\sum_{c}F^{c\,\pi(c)}(G)$
faces since its faces are precisely the faces of color $\{c,\pi(c)\}$ of
$G$. Euler's formula yields the relation
$2k(G)+\sum_c F^{c\,\pi(c)} (G) =(D+1)k(G)+2-2g_\pi$, where $g_\pi$ denotes the
genus of $G_\pi$. Taking into account that each face of $G$ (say with colors $\{c,c'\}$) is 
counted by $2(D-1)!$ cycles $\pi$ (those such that $\pi(c)=c'$ and those such that $\pi(c')=c$), we have 
$\sum_{\pi} \sum_c F^{c\,\pi(c)} (G) = 2(D-1)! F(G)$, and upon summing over all $\pi$ we obtain 
$\delta(G)=\frac{1}{(D-1)!}\sum_{\pi} g_\pi$. The positivity of $\delta(G)$ thus
follows from that of the genera of all jackets.

\qed
\end{proof}

Observe that the degree of a ring graph is zero, $\delta(R^c)=0$.
The following corollary summarizes the relations we shall need on
colored graphs:

\begin{corollary}
Let $G$ be a nontrivial colored graph with $2k(G)$ vertices, $(D+1)k(G)$ edges, degree $\delta(G)$ and $F(G)$ faces, $F_p(G)$ of which have length $2p$. Then:
\bea\label{eq:doublecount1} 
D(D+1)k(G) &=&  2\sum_{p\ge 1} pF_p (G) \; , \\\label{eq:faces}
D(D-1)k(G) &=& 2F(G) +2\delta(G) -2D \; , \\\label{eq:delta'} 
(D+1)\delta (G) + 2F_1(G) &=& D(D+1)+ \sum_{p\geq
    2} \left[(D-1)p-D-1 \right]F_p(G) \; . 
\eea
\end{corollary}
\begin{proof}
The first equation follows by observing that in the graph $G$ there are a total of $D(D+1) k(G)$ pairs of edges at vertices of $G$, and that each such pair is 
a corner on some face. The second is the definition of the degree.  The third one is a simple linear combination of the first
two. 

\qed
\end{proof}

Equation~(\ref{eq:delta'}) already suggests that the classical case
$D=2$, which we do not consider here, is very different from
$D\geq4$. On the one hand when $D=2$, the coefficient of $F_2(G)$ in the
right hand side is negative, so that the number of large faces (or the
degree of the largest face) can be arbitrarily large even if $\delta(G)$
and $F_1(G)$ are fixed. On the other hand when $D\geq4$, the right hand
sum has only positive coefficients and the number of large faces (or
the degree of the largest face) is bounded by $(D+1)\delta(G)+2F_1(G)$. The
case $D=3$ is \emph{a priori} intermediate in the sense that faces of
degree 4 could proliferate at fixed $\delta(G)$ and $F_1(G)$ (since the
coefficient of $F_2(G)$ is zero in this case), but we shall see later
that this does not happen, and the dichotomy is really between $D=2$
and $D\geq 3$.

\newpage

\section{The core of a rooted colored graph}\label{sec:core}

\subsection{First attempts to define the core}\label{ssec:mremove}

A \emph{melon with external color $c$} (or simply a \emph{melon}) $\mathbb{O}^c$ in a colored graph $G$ is an open colored subgraph of $G$ 
consisting of $D$ parallel edges joining two vertices, and two
half-edges of color $c$ (one on each vertex of $\mathbb{O}^c$). Depending on the external color $c$, there are $D+1$ distinct types of melons. 
The \emph{melon removal} of $\mathbb{O}^c$ in $G$ consists in cutting the two edges of $G$ corresponding to the half-edges 
of $\mathbb{O}^c$, deleting $\mathbb{O}^c$, and gluing the two remaining half-edges to recreate an edge of color $c$ (see Fig.~\ref{fig:meloremove}).

\begin{figure}[ht]
\begin{center}
\psfrag{m}{$\mathbb{O}^c$}
\psfrag{c}{$c$}
\includegraphics[scale=.3]{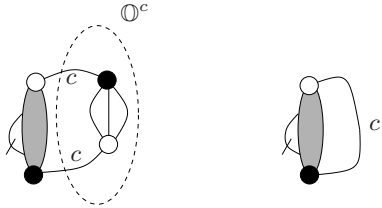}
\end{center}
\caption{The removal of a melon}\label{fig:meloremove}
\end{figure}

We would like to define the \emph{core} of a rooted colored graph
as ``the'' graph obtained by a maximal sequence of melon
removals. However in this ``definition'', it is not clear that the
core is uniquely defined.  We will therefore follow an alternative
approach, which is reminiscent of an old result of Tutte \cite{tutte}
about \emph{2-connected graphs}, that is, graphs that cannot be
disconnected by removing a single vertex.  In the modern phrasing of
\cite{fusy}, Tutte's result states that any 2-connected graph $G$ with set of vertices $\cV$ and set of edges $\cE$ (which we denote
$G=(\cV,\cE)$) can be decomposed into an arborescent structure called its
RMT-tree by a careful analysis of all its 2-cuts, that is, its pairs
$\{x,y\}$ of vertices such that $G|_{\cV\setminus\{x,y\}}$ is not
connected.  Our interest in this result arises from the following
lemma.

\begin{lemma}
Non trivial colored graphs are 2-connected.
\end{lemma}
\begin{proof}
Consider a colored graph $G=(\cV,\cE)$ containing a white vertex $w_{\circ}$ such that
$G'=G|_{\cV\setminus\{w_{\circ}\}}$ is not connected and let $G_1$ be a connected component of
$G'$. Let $\ell$ denote the number of edges between $w_{\circ}$ and
vertices of $G_1$. By hypothesis there is at least one other connected
component in $G'$ so that $1\leq \ell\leq D$. Let us consider the
subgraph of $G$ induced by $w_{\circ}$ and the vertices of $G_1$.  In this
subgraph all the black vertices and all the white vertices except $w_{\circ}$ have
degree $D+1$, while $w_{\circ}$ has degree $1\leq\ell\leq D$: double counting
of edges yields a contradiction.

\qed
\end{proof}

In view of this lemma one could directly apply Tutte's result to
colored graphs and obtain a decomposition by describing the resulting
RMT-trees, and this was our approach in an earlier version of this
article. However it turns out to be easier to derive the decomposition
we need from a direct analysis of 2-edge-cuts.

\subsubsection{2-edge-cuts in colored graphs.}
A \emph{2-edge-cut} of a 2-connected graph $G=(\cV,\cE)$ is a pair of
edges $\{e,e'\}$ such that the graph
$G-\{e,e'\}=(\cV,\cE\setminus\{e,e'\})$ is not connected (see
Fig.~\ref{fig:2-cuts} left hand side).  Again by double counting of
edges, one easily checks that the two edges of a 2-edge-cut in a
colored graph must have the same color.  A \emph{simple cycle} of a
graph $G$ is a cycle visiting each vertex of $G$ at most once.

\begin{lemma}\label{lem:2-cut}
Let $G$ be a 2-connected graph. Then $\{e,e'\}$ is a 2-edge-cut in $G$
if and only if any simple cycle visiting $e$ also visits $e'$.
\end{lemma}
\begin{proof}
Let $e=\{x,y\}$ and let $Cycle$ be a simple cycle visiting $e$. Then
$G-\{e,e'\}$ has two connected components, one containing $x$ and the
other containing $y$. But $Cycle-e$ is a path from $x$ to $y$ and $e'$ is
the only edge connecting the two components in $G-\{e\}$.

Conversely if $\{e,e'\}$ is not a 2-edge-cut then there exists a path
in $G-\{e,e'\}$ between any  two vertices, and in particular
between the endpoints of $e$. 

\qed
\end{proof}

\begin{lemma}\label{lem:cyclic}
Let $\{e,e'\}$ and $\{e,e''\}$ be two distinct 2-edge-cuts in a 2-connected
graph $G$.  Then $\{e',e''\}$ is also a 2-edge-cut of $G$. Moreover if
two oriented cycles visit $e$ in the same direction, then they both
visit $e'$ and $e''$ in the same order after $e$.
\end{lemma}
\begin{proof}
Any cycle visiting $e'$ visits also $e$ (since $\{e,e'\}$ is a
2-edge-cut) and thus also $e''$ (since $\{e,e''\}$ is a 2-edge-cut),
and we conclude by Lemma~\ref{lem:2-cut}. Next assume that there exist
two oriented cycles $(e,Path_1,e', Path_1', e'',Path_1'')$ and
$(e,Path_2,e'',Path_2',e',Path_2'')$ that visit $e$ in the same
direction. But then the path $Path_2$ would connect the two connected
components of $G-\{e,e'\}$ without visiting $e$ or $e'$.

\qed
\end{proof}

A \emph{proper cut-set} of a 2-connected graph $G$ is a maximal set
$Cut$ of edges such that any 2 edges of $Cut$ form a 2-edge-cut. In view
of the previous lemma, an edge can belong to at most one proper
cut-set, so that we define the \emph{cut-set of an edge} as the unique
proper cut-set it belongs to if it exists, or the edge itself otherwise.

\begin{lemma}
Let $G=(\cV,\cE)$ be a non trivial colored graph and let $e_0$ be an edge of $G$ with non-trivial cut-set $Cut$.  Then
there exists a unique way to cyclically arrange the edges of $Cut$ as
$(e_0,\ldots,e_\ell)$ and a unique partition $\cV_0,\ldots,\cV_\ell$ of
$\cV$ such that $\cE=Cut\cup \cE_{\cV_0}\cup\ldots\cup \cE_{\cV_\ell}$, where $\cE_{\cV_i}$ are edges connecting only vertices in $\cV_i$,
and for all $i=0,\ldots,\ell$, $e_{i}$ connects a black vertex of $\cV_i$ to a white
vertex of $\cV_{i+1}$ with indices taken modulo $\ell+1$ (see Fig.~\ref{fig:2-cuts}).
\end{lemma}
\begin{proof}
 Since $G$ is 2-connected, there exists a
simple cycle visiting $e_0$, and this cycle also visits the other
edges in $Cut$. Orient this cycle so that $e_0$ is visited from its
black to its white endpoint and let $(e_0,e_1,\ldots,e_\ell)$ describe
the cyclic arrangement of the edges of $Cut$ along this oriented
cycle. Then $\{e_i,e_{i+1}\}$ forms a 2-edge-cut and one of the
components of $G-\{e_i,e_{i+1}\}$ contains the edges 
$e_{i+2},\ldots,e_{i-1}$. Let $\cV_i$ denote the vertex set of other
component. Then the $\cV_i$ are disjoint and form a partition of $\cV$ and
the other required properties are immediate. The uniqueness follows
from the fact that any other cycle traversing $e_0$ has to visit the edges of
$Cut$ in the same order (Lemma~\ref{lem:cyclic}).

\qed
\end{proof}

We will be especially interested in the cut-set of the root edge $r(G)$ of a colored graph $G$.
The list $( \mathbb{G}_0,\ldots, \mathbb{G}_\ell)$ of \emph{open components} of the cut-set of the root $r(G)$ is the set of open colored graphs
obtained by cutting each edge of this cut-set into two half-edges: with
the notation of the lemma, the vertex set of $\mathbb{G}_i$ is $\cV_i$ and the
two half-edges of $\mathbb{G}_i$ arise from $e_i$ and $e_{i+1}$ respectively. The
set of \emph{closed components} of the cut-set of the root is then the set
of colored graphs obtained from the open components by
reconnecting in each of them the unique available pair of half-edges
to form a root edge: $ \left( \textrm{cl}( \mathbb{G}_i), i=0,\ldots,\ell \right)$.
These definitions are illustrated by Fig.~\ref{fig:2-cuts}.

\begin{figure}[ht]
\begin{center}
\psfrag{e}{$ e $}
\psfrag{ep}{$e'$}
\psfrag{e0}{$r(G)=e_0$}
\psfrag{e1}{$e_1$}
\psfrag{e2}{$e_2$}
\psfrag{e3}{$e_3$}
\psfrag{V0}{$\cV_0$}
\psfrag{V1}{$\cV_1$}
\psfrag{V2}{$\cV_2$}
\psfrag{V3}{$\cV_3$}
\psfrag{G0}{$\mathbb{G}_0$}
\psfrag{G1}{$\mathbb{G}_1$}
\psfrag{G2}{$\mathbb{G}_2$}
\psfrag{G3}{$\mathbb{G}_3$}
\psfrag{cG0}{$\textrm{cl}(\mathbb{G}_0)$}
\psfrag{cG1}{$\textrm{cl}(\mathbb{G}_1)$}
\psfrag{cG2}{$\textrm{cl}(\mathbb{G}_2)$}
\psfrag{cG3}{$\textrm{cl}(\mathbb{G}_3)$}
\includegraphics[scale=.5]{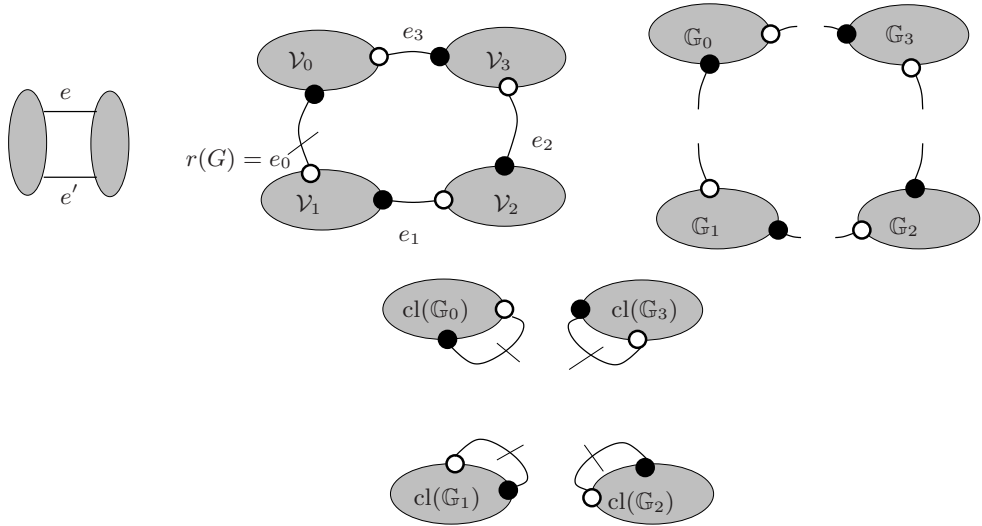}
\end{center}
\caption{A 2-edge-cut, the set-cut of the root edge and the corresponding open and
  closed components.}\label{fig:2-cuts}
\end{figure}

Finally the following immediate lemma will be useful.

\begin{lemma}\label{lem:useful}
Let $Cut'$ and $Cut''$ be two distinct cut-sets in a colored graph
$G$. Then there is an open component $\mathbb{G}'$ of $Cut'$ containing $Cut''$,
and an open component $\mathbb{G}''$ of $Cut''$ containing $Cut'$. 
\end{lemma}
\begin{proof}
Assume $Cut''$ is not contained in any of the open components of $Cut'$,
and let $e''_i,e''_j$ be two edges of $Cut''$ appearing in two different
components $\mathbb{G}'_p$ and $\mathbb{G}'_q$ of $Cut'$. Then any cycle visiting $e''_i$
also visits $e''_j$, and thus $e'_p$ and $e'_{p+1}$ and so that $Cut'$
and $Cut''$ are not disjoint. 

\qed
\end{proof}

\subsection{Melons and melonic graphs}

An important ingredient in our approach are the open \emph{melonic graphs}, defined inductively (see Fig.~\ref{fig:melons}) below.

\begin{definition}\label{def:melonin}
 An open colored graph $\mathbb{M}$ with half-edges of color $c$ is a:
\begin{itemize} \item  \emph{melonic graph} if:
\begin{itemize}
  \item either it is trivial, that is $\mathbb{M} = \mathbb{E}^c= \mathrm{op}(R^c)$,
 \item or all the open components of the cut-set of the root edge of $\mathrm{cl}(\mathbb{M} )$ are \emph{prime melonic graphs}. 
\end{itemize}
\item \emph{a prime melonic graph} if it is non trivial and by cutting \emph{all} the edges incident to the vertices $v_{\circ}$ and $v_{\bullet}$ of $\mathbb{M}$ which carry the external 
   half-edges of $\mathbb{M}$ into pairs of half-edges and subsequently deleting $v_{\circ}$, $v_{\bullet}$ and all the half-edges attached to them
   one obtains $D$ (one for each $c'\neq c$) open \emph{melonic graphs}.
   Observe that if $v_{\circ}$ and $v_{\bullet}$ are connected by an edge in $\mathbb{M}$, then this edge is cut twice and this creates a trivial melonic graph.
\end{itemize}

A colored graph $G$ is \emph{melonic} if $\mathrm{op}(G)$ is an open melonic graph. In particular a ring graph $R^c$ is melonic.
\end{definition}

\begin{figure}[ht]
\begin{center}
\psfrag{c}{$c$}
\psfrag{x}{$v_{\circ}$}
\psfrag{y}{$v_{\bullet}$}
\psfrag{c1}{$c'$}
\includegraphics[width=10cm]{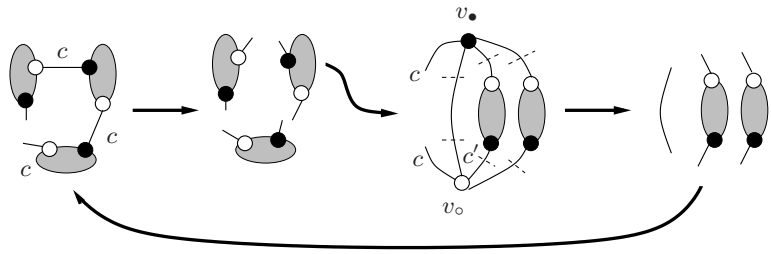}
\end{center}
\caption{Illustration of the definition of melonic graphs.}\label{fig:melons}
\end{figure}

In the previous literature (see \eg \cite{color,lost}) melonic graphs
have been defined as the graphs that can be obtained from some ring
graph $R^c$ by cutting edges and inserting recursively melons
$\mathbb{O}^{c'}$.  Let us now recall the (known) consistence of this
 definition with ours. Let $G[e^c\leftarrow \mathbb{O}^{c } ]$ denote the
rooted colored graph obtained by the \emph{insertion} of a melon
$\mathbb{O}^{c }$ at a non-root edge $e^c$ (of color $c$) of $G$, \ie
by cutting $e^c$ into two half-edges and gluing these half-edges to
those of $\mathbb{O}^{c }$ in the unique way that makes the result
bipartite. The insertion at the root edge $r(G)=\{
v_{\circ},v_{\bullet} \}$ of $G$ can be performed in two ways: either
such that the root of $G[r(G)\leftarrow \mathbb{O}^{c }]$ is hooked to
the white vertex $v_{\circ}$ or such that the new root edge is hooked
to the black vertex $v_{\bullet}$.

\begin{proposition}\label{pro:extract}
A colored graph with root edge of color $c$ is melonic if and only if it can be obtained from the ring graph $R^c$ by a sequence of melon insertions.
\end{proposition}
\begin{proof} By induction, observing that the only prime melonic open graphs with two vertices are the $\mathbb{O}^{c }$s.

\qed
\end{proof}

Our interest in melonic graphs arises from the following two known
results (\emph{e.g.} \cite{color,lost}): 
\begin{proposition}\label{pro:degree}
  The degree of $G$ and $G[e^c\leftarrow \mathbb{O}^{c }]$ are the same.
\end{proposition}
\begin{proof}
 By construction $G[e^c\leftarrow \mathbb{O}^{c }]$ has two more vertices and
 $\binom{D}{2}$ more faces than $G$, hence by Equation~(\ref{eq:faces}) they have the same degree.
 
\qed
\end{proof}
\begin{theorem}\cite{color,lost} A rooted colored graph $G$ is melonic if and only if it has degree 0. 
\end{theorem}
\begin{proof}
 Propositions~\ref{pro:extract} and~\ref{pro:degree} and the remark that the reduce degree of a ring graph is zero immediately imply that melonic graphs
 have degree 0. 
 
 In order to prove the other implication, observe that if $G$ is non trivial and $\delta(G)=0$, then from Eq.~\eqref{eq:delta'} it ensues that $F_1(G)\ge D(D+1)/2$, hence there exists a 
 face with colors $(c_1,c_2)$ of length two (\ie formed by the two edges $e^{c_1}=\{v_{\circ},v_{\bullet}\}$ and $e^{c_2}=\{v_{\circ},v_{\bullet}\}$) of $G$.
 
 If for all $c\neq c_1,c_2$ the two vertices are connected by a unique edge $e^c=\{v_{\circ},v_{\bullet}\}$, then  $G=R^{c'}[e^{c'} \leftarrow \mathbb{O}^{c' }]$ and $G$ is melonic.
 Else, there exists $c$ such that the edge of color $c$ hooked to $v_{\circ}$ (denoted $e_1^c$) is different from the edge of color $c$ hooked to $v_{\bullet}$ (denoted $e^c_2$). 
 As $\delta(G)=0$ all the jackets $G_{\pi}$ of $G$ are planar, including $\pi(c_1)=c, \pi(c) = c_2$, hence $\{ e_1^c,e^c_2 \}$ is a 2-edge-cut of $G$.
 Cutting $e^c_1$ and $e^c_2$ into pairs of half-edges $e^c_1=\Braket{h^1_{\circ} , h^1_{\bullet}}$ and $e^c_2=\Braket{ h^2_{\circ} , h^2_{\bullet}  }$
 and reconnecting the half-edges the other way around into two new edges $f_1 = \Braket{ h^1_{\circ} ,  h^2_{\bullet} } $ and 
 $f_2 = \Braket{ h^2_{\circ} , h^1_{\bullet} }$, $G$ splits into two colored graphs $G_1$ (containing the edge $f_1$) and $G_2$ (containing the edge $f_2$). Some care must be taken with the root: 
\begin{itemize}
 \item  if $r(G)\neq e^c_1,e^c_2$ and $r(G)$ belongs let's say to $G_1$, then $r(G_1) = r(G)$ and $r(G_2)=f_2$,
 \item if $r(G) =e^c_1$ or $r(G)=e^c_2$ then $r(G_1)=f_1$ and $r(G_2)=f_2$.
\end{itemize}

As $k(G_1) + k(G_2) =k(G)$ and $F(G_1) + F(G_2) = F(G) + D$, it follows that $\delta(G_1) = \delta(G_2)=0$ and taking into account that $k(G_1),k(G_2)<k(G)$, the  theorem follows by induction.
 
 \qed
\end{proof}

\subsection{The core decomposition}

Two open colored subgraphs of an open colored graph $\mathbb{G}$ are
\emph{totally disjoint} if their edge sets are disjoint and their
half-edges belong to different edges of $\mathbb{G}$. We shall be
particularly interested in open colored subgraphs that are open
melonic graphs: for short we refer to these as \emph{open melonic
  subgraphs}. Recall that open colored subgraphs are assumed non
trivial, and so are open melonic subgraphs.

The following lemma is illustrated in Fig.~\ref{fig:non-disjoint}.

\begin{lemma}\label{lem:union}
If two open melonic subgraphs of a rooted colored graph $G$ are not
totally disjoint then their union is an open melonic subgraph of $G$.
\end{lemma}
\begin{proof}
First observe that if $\mathbb{M}$ is an open melonic subgraph of $G$
with half-edges $h_{\circ}\in e$ and $h_{\bullet}\in e'$, then the edges $e$ and $e'$
of $G$ form a 2-edge-cut of $G$, unless $e=e'=r(G)$, in
which case $\mathbb{M}=\mathrm{op}(G)$. Indeed, since each vertex of $\mathbb{M}$
has the same degree (counting the half-edges) in $\mathbb{M}$ and in $G$, all the edges of $G$ incident
to a vertex of $\mathbb{M}$ are in $\mathbb{M}$, except for $e$ and $e'$.

Now if $\mathbb{M}_1$ and $\mathbb{M}_2$ are two open melonic subgraphs of $G$ with
half-edges $h^1_{\circ}\in e_1$ and $h^1_{\bullet}\in e_1'$, and $h^2_{\circ}\in e_2$ and
$h^2_{\bullet}\in e_2'$ respectively, then the two cuts $\{e_1,e_1'\}$ and
$\{e_2,e_2'\}$:
\begin{itemize}
\item either belong to two different cut-sets: in this case, in view of
  Lemma~\ref{lem:useful}, either $\mathbb{M}_1\subset \mathbb{M}_2$ or $\mathbb{M}_2\subset \mathbb{M}_1$
  (or $\mathbb{M}_1$ and $\mathbb{M}_2$ are totally disjoint but this contradict the
  hypothesis).
\item or belong to the same cut-set: in this case, since $\mathbb{M}_1$ and
  $\mathbb{M}_2$ are not totally disjoint, there are two edges $e''_1\in
  \{e_1,e_1'\}$ and $e''_2\in\{e_2,e_2'\}$ such that $\mathbb{M}_1\cup \mathbb{M}_2$ is
  the component of $G-\{e''_1,e''_2\}$ not containing the root. In
  particular the closed components of the cut-set of the root of
  $\textrm{cl}(\mathbb{M}_1\cup \mathbb{M}_2)$ are the union of the closed components of
  the cut-sets of the roots of $\textrm{cl}(\mathbb{M}_1)$ and
  $\textrm{cl}(\mathbb{M}_2)$ and they are all prime melonic graphs, so that
  $\mathbb{M}_1\cup \mathbb{M}_2$ is melonic.
\end{itemize}
 \qed
\end{proof}

\begin{figure}[ht]
\begin{center}
\includegraphics[scale=.2]{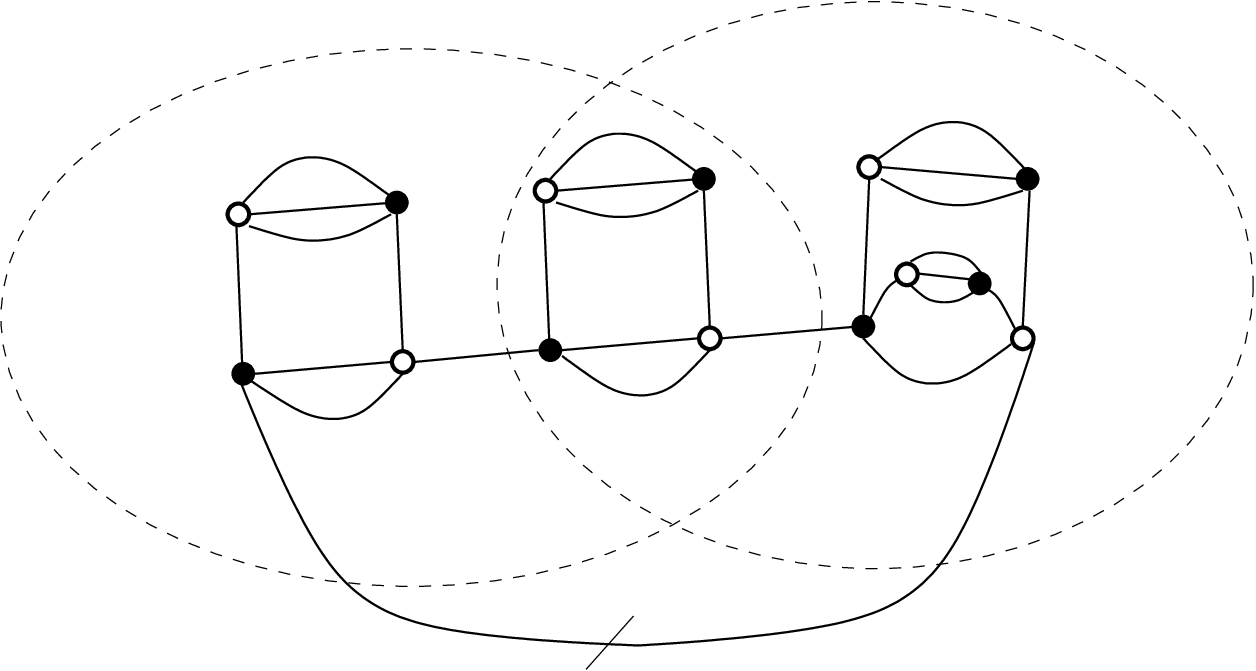}
\end{center}
\caption{Two non-disjoint melonic subgraphs of a colored graph.}\label{fig:non-disjoint}
\end{figure}

A \emph{maximal melonic subgraph} of a colored graph $G$ is a (non trivial) open melonic subgraph which is maximal for inclusion.
\begin{lemma}
The maximal  melonic subgraphs of $G$ are totally disjoint.
\end{lemma}
\begin{proof}
This is an immediate consequence of Lemma~\ref{lem:union}.

\qed
\end{proof}

\begin{definition}
The \emph{core} of an open colored graph $\mathbb{G}$ is the open colored graph $\hat {\mathbb{G}}$
obtained from $\mathbb{G}$ by deleting each of its maximal melonic
subgraphs and gluing the resulting pairs of half-edges to reform
edges. The core $\hat G$ of a colored graph $G$ is the colored graph obtained by deleting each of the maximal melonic subgraphs of $G$, i.e. 
$\hat G$ is the closure of the core of $\mathrm{op}(G)$.
\end{definition}

Observe that if one of the two half-edges of $\mathbb{G}$ belongs to a
melonic subgraph then this half-edge is considered to be cut and
re-glued.  If both half-edges of $\mathbb{G}$ belong to the same
melonic subgraph, then $\mathbb{G}$ is melonic and the core is a
trivial one edge graph $\mathbb{E}^c$.

A colored graph is \emph{melon-free} if it does not contains a melonic subgraph.
Observe that according to our conventions, the ring graphs $R^c$ are considered melon-free: in view of Proposition~\ref{pro:extract} they are in fact 
the only melon-free graphs of degree 0. By construction the core of an open colored graph $\mathbb{G}$ (or of a colored graph $G$) is melon-free.

The following characterization (which is in fact not used in the rest
of the paper) more generally relates our construction to the initial
discussion of this section and is a direct consequence of Proposition~\ref{pro:extract}.
\begin{proposition}
The core of a colored graph $G$ is the unique melon-free graph that can be obtained from $G$ by a sequence of melon removals (discussed in Section~\ref{ssec:mremove}).
\end{proposition}

By definition of the core $\hat G$ of a colored graph $G$, for
each non-root edge $e=\{w_{\circ},w_{\bullet}\}$ with color $c$ in $\hat G$, 
\begin{itemize}
\item either
there is a maximal melonic subgraph in $G$, which we denote $\mathbb{M}_e$, whose half-edges have
color $c$ and respectively point to $w_{\circ}$ and $w_{\bullet}$, 
\item or there is an edge
$\{w_{\circ},w_{\bullet}\}$ with color $c$ in $G$ and this edge is not involved in any
melonic subgraph of $G$, and in this case we set $\mathbb{M}_e=\mathbb{E}^c$ by
convention.  
\end{itemize}
Considering the root edge $r(\hat G)= \{ v_{\circ}, v_{\bullet}\}$ of a nontrivial core $\hat G$ to be formed of two matched half-edges $h_{\circ}$ and $h_{\bullet}$, 
for each of these half-edges:
\begin{itemize}
\item either there is a maximal melonic subgraph denoted $\mathbb{M}_\circ$ (resp. $\mathbb{M}_\bullet$) whose white
  (resp. black) half-edge is $h_{\circ}$ (resp. $h_{\bullet}$)
\item or the root edge of $G$ is hooked directly to $v_{\circ}$ (resp. $v_{\bullet}$), and in this case we set $\mathbb{M}_{\circ}=\mathbb{E}^c$ (resp. $\mathbb{M}_\bullet=\mathbb{E}^c$).
\end{itemize}
If the core is trivial, that is $\hat G=R^c$, then $G$ is melonic:
$\mathrm{op}(G)$ is then the (possibly trivial) melonic subgraph
$\mathbb{M}_{r(R^c)}$ associated to the unique edge
$r(R^c)$ of $R^c$.

The \emph{core decomposition} of a colored graph $G$ is the tuple $(\hat G; (\mathbb{M}_\circ, \mathbb{M}_\bullet, \mathbb{M}_{e_2},\ldots, \mathbb{M}_{e_{(D+1) k(\hat G) }}))$, where
$e_2,\ldots,e_{(D+1) k(\hat G) }$ is a canonical list of the edges of $\hat G$ excluding the root edge.
The following theorem, exemplified in Fig.~\ref{fig:core}, summarizes and is an immediately consequence
of the above discussion.

\begin{figure}[ht]
\begin{center}
\includegraphics[width=10cm]{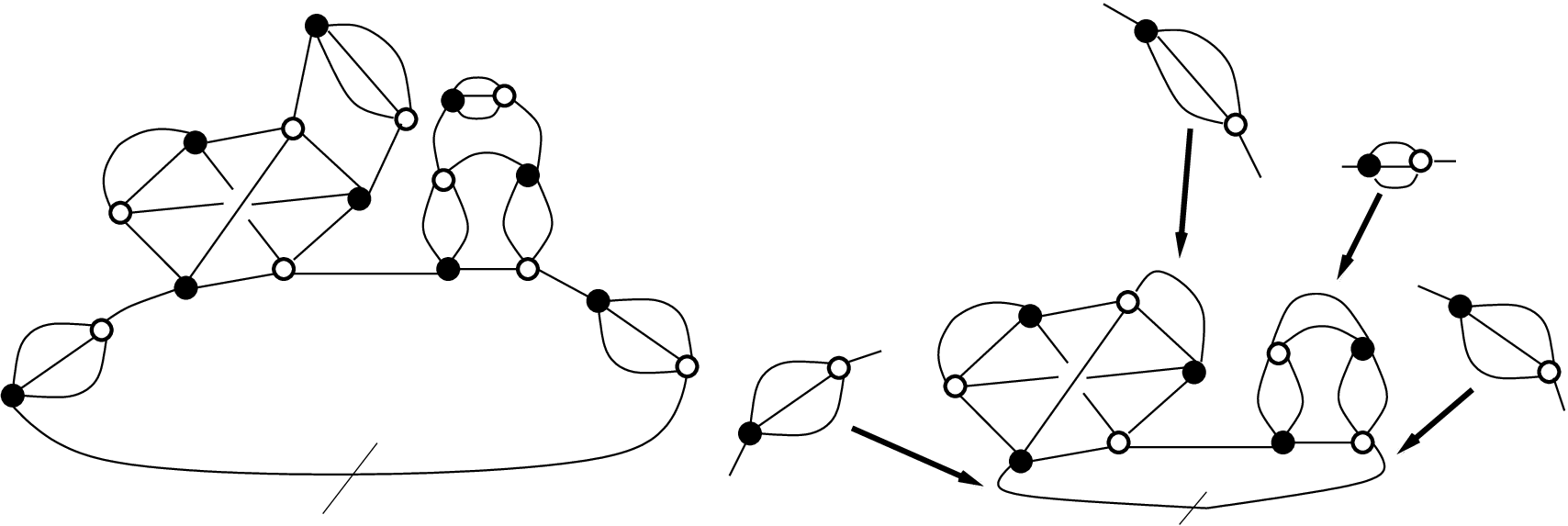}
\end{center}
\caption{The maximal melonic subgraphs of a colored graph, and its core decomposition (only the nontrivial melonic subgraphs are represented).}\label{fig:core}
\end{figure}

\begin{theorem}[Core decomposition]\label{thm:core}
The core decomposition is one-to-one  between
\begin{itemize}
\item colored graphs $G$ with $2k(G)$ vertices and degree $\delta(G)$,
\item pairs $ (\hat{G}; (\mathbb{M}_\circ, \mathbb{M}_{\bullet},\mathbb{M}_2 \ldots, \mathbb{M}_{(D+1)k( \hat{G} )}))$ where 
   \begin{itemize}
   \item $\hat{G}$ is a melon-free colored graph with $2k(\hat{G})$ vertices and degree $\delta( \hat{G})=\delta(G)$, and 
   \item for all $i \in \{ \circ, \bullet, 2 ,\ldots,(D+1)k( \hat{G} ) \} $, $\mathbb{M}_i$ is a possibly trivial open
     melonic graph with $2 k(\mathbb{M}_i)$ vertices,
   \end{itemize}
   such that  $2 k( \hat{G}) + \sum_{i } 2    k(\mathbb{M}_i ) =2k (G) $.
\end{itemize}
\end{theorem}
\begin{proof}
The core decomposition is clearly injective since all the components
of $G$ have been kept in the decomposition as well as the
correspondence between edges of the core and subgraphs. Conversely any
such pair yields a rooted colored graph $G$ by substitution of the $ \mathbb{M}_i$
in $ \hat{G}$ and all substituted melonic subgraphs become totally disjoint
maximal melonic subgraphs in $G$, so that $\hat{G}$ is the core of $G$ and 
the core decomposition gives back the $ \mathbb{M}_i$.

The relation between $k( \hat{G})$, $k(\mathbb{M}_i )$ and the $k (G)$ follows from the remark
that the core decomposition is a partition of the vertex set of $G$,
while the fact that a rooted colored graph and its core have the same  
degree follows from Proposition~\ref{pro:degree}.

\qed
\end{proof}

\newpage

\section{Chains}\label{sec:chains}
\subsection{Maximal chains}
The set of cores of fixed degree is not finite.
This is due to the presence of \emph{chains} of $(D-1)$-dipoles 
(to be defined precisely below) of arbitrary length. It follows 
that, in order to provide a useful classification of graphs 
at fixed degree, we need to refine further the 
core decomposition by identifying and removing maximal chains.
The definition below is slightly different from the one found
in \cite{FG}.

\begin{definition}
A \emph{$(D-q)$-dipole} in a colored graph $G$ is a couple of vertices connected by exactly
$D-q$ parallel edges, \emph{none of which is the root} of the graph. 
A $(D-q)$-dipole is attached to the rest of the graph by $q+1$ pairs of half-edges of the same color. 
It is possible that one of these half-edges belongs to the root edge $r(G)$ or that 
two of these half-edges are matched to form the root edge $r(G)$ (see Fig.~\ref{fig:clardipole} below).

The parallel edges of a dipole are called \emph{internal edges} of the dipole.
The edges of $G$ to which the half-edges of a dipole belong are called \emph{external edges} of the dipole.

A  $(D-1)$-dipole is said to \emph{have external colors $ ( c_1,c_2 ) $} if a pair of half-edges 
has color $c_1$ and the other one has color $c_2$. 
\end{definition}
 
\begin{figure}[ht]
\begin{center}
\psfrag{c1}{$c_1$}
\psfrag{c2}{$c_2$}
\includegraphics[scale=0.5]{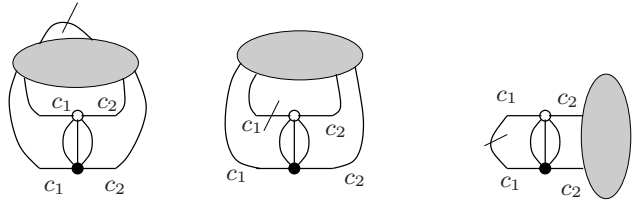}
\end{center}
\caption{Examples of $(D-1)$-dipoles for $D=4$.} \label{fig:clardipole}
\end{figure}

Although dipoles are defined for arbitrary colored graphs, in the
rest of the text we will only be interested in dipoles of melon-free colored
graphs.

Dipoles can join together to form chains of dipoles (as in Figure~\ref{fig:cchains}). The chains of $(D-1)$-dipoles are especially important: as we 
will see below, the degree of a graph does not depend of the length of such chains hence all the cores which only differ by the length of
the internal chains of $(D-1)$-dipoles have the same degree. 

\begin{figure}[ht]
\begin{center}
\psfrag{lc}{$\ell_{\circ}$}
\psfrag{ld}{$\ell_{\bullet}$}
\psfrag{rc}{$r_{\circ}$}
\psfrag{rd}{$r_{\bullet}$}
\psfrag{ub2p+1}{$\underbrace{{}\qquad \qquad{}}_{p=2s+1}$}
\psfrag{ub2p}{$\underbrace{{} \qquad \qquad{}}_{p=2s}$}
\includegraphics[scale=0.5]{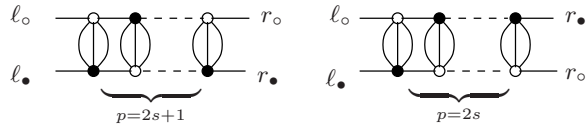}
\end{center}
\caption{Two chains, with odd and even length respectively (with
  $D=4$).} \label{fig:cchains}
\end{figure}

One would like to identify the maximal vertex disjoint chains of $(D-1)$-dipoles in a core. The case $D=3$ is slightly subtle 
and requires a refinement of the naive definition of a chain.

\begin{definition}
A \emph{chain with external colors $(c_1,c_2 )$} (which can coincide) in a melon-free colored graph $G$
is a connected sub-pre-graph  (see Figure~\ref{fig:cchains}) made of:
\begin{itemize}
\item two \emph{left half-edges} $\ell_\circ$ and $\ell_\bullet$ \emph{having the same color} $c_1$,
\item two \emph{right half-edges} $r_\bullet$ and $r_\circ$ \emph{having the same color} $c_2$,
\item $2p$ \emph{internal vertices}, $p\geq1$, forming a sequence
  $d_1,\ldots,d_p$ of $(D-1)$-dipoles,
\end{itemize}
such that:
\begin{itemize} 
\item the white (resp. black) vertex of $d_1$ is incident to $\ell_\circ$
(resp. $\ell_\bullet$),
\item the white (resp. black) vertex of $d_p$ is incident to $r_\circ$
(resp. $r_\bullet$),
\item two half-edges of the dipole $d_i$ are joined to two half-edges of the dipole $d_{i+1}$ for each $i=1,\ldots,p-1$,
\item the root of $G$ is \emph{not one of the internal edges of the chain} (\ie neither an internal edge of a dipole, nor an edge connecting two consecutive dipoles).
Observe however that the root edge can contain one of the half-edges $\ell_\circ, \ell_\bullet,  r_{\circ}, r_\bullet$, or 
it can consist in \emph{any matching} of a pair of half-edges: $ \Braket{ \ell_\circ , \ell_\bullet }, \Braket{ \ell_\circ , r_\bullet } , \Braket { \ell_\bullet , r_\circ }$ or $\Braket{ r_\circ , r_\bullet} $
(see Fig.~\ref{fig:identifychains}).
\item the half-edges  $\ell_\circ,\ell_\bullet$ and respectively $r_\circ,r_\bullet$ are \emph{not matched together in a non-root edge of $G$} 
(\ie the two half-edges at the same end of the chain do not belong to the same non-root edge in $G$, 
as in this case the chain would be a melonic subgraph). 
On the contrary, $\ell_\circ$ and $r_\bullet$ and respectively $\ell_\bullet$ and $r_\circ$ can be matched together (see Fig.~\ref{fig:identifychains}), in a non-root edge.
\end{itemize}

A chain is \emph{proper} if it contains at least 4 internal vertices (or equivalently at least two
$(D-1)$-dipoles). A chain is \emph{maximal} if it is not contained in a larger chain.

\end{definition}

\begin{figure}[ht]
\begin{center}
\includegraphics[scale=.4]{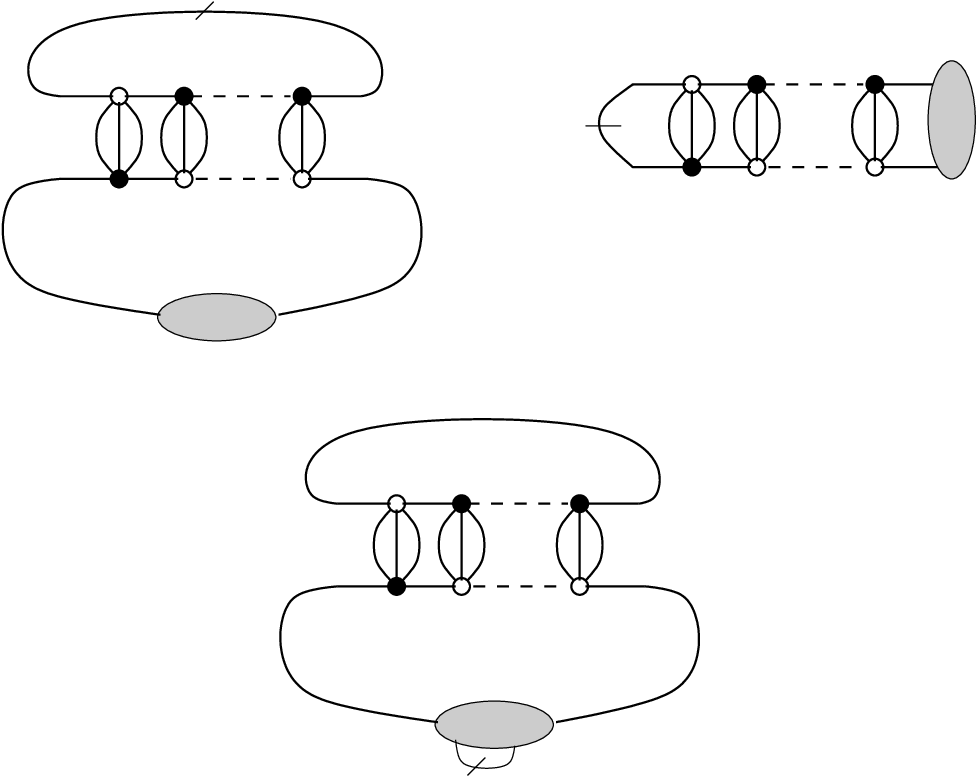}
\end{center}
\caption{Chains with half-edges matched.} \label{fig:identifychains}
\end{figure}

With this definition we have the following important results.

\begin{lemma}\label{lem:uniqueextension}
 Each chain of $G$ can be extended in a unique way to a maximal chain.
\end{lemma}

\begin{proof}
 Let us consider a chain and its two left half-edges $\ell_\circ$ and $\ell_\bullet$. If the two half-edges are incident to 
 a $(D-1)$-dipole in $G$ and neither of the two belongs to the root edge $r(G)$, then the chain can 
 be \emph{extended to the left} by adding the $(D-1)$-dipole. The chain can similarly be \emph{extended to the right}.
 The crucial point is that in order to decide whether a chain can be extended, one only tests the half-edges at the same end of the chain.
 
Extending maximally the chain, one obtains the same maximal chain, \emph{irrespective of the order} in which the left/right extensions are performed.

\qed
\end{proof}

A comment is important at this point. Observe that Lemma~\ref{lem:uniqueextension} works because we have allowed a left and right 
half-edge to be matched in a non-root edge, but it would not work otherwise. Indeed, consider the case at the bottom of Fig.~\ref{fig:identifychains}. 
If the chain can be extended only if the left and right half-edges are not matched together, then one obtains two  different 
maximal extensions (including respectively the last dipole on the left or on the right) of the same chain.

\begin{lemma}
  If $D\geq 4$, two distinct maximal chains in a melon-free 
  colored graph $G$ cannot share an internal vertex.  If $D=3$, two
  distinct maximal proper chains in a melon-free colored
  graph $G$ cannot share an internal vertex.  
\end{lemma}
\begin{proof}
  Assume first that the rooted colored graph $G$ contains two maximal chains
  that share a $(D-1)$-dipole. But then these chains are the maximal left/right extension 
  of the dipole, hence coincide. 

  Now, if two chains share a vertex but no $(D-1)$-dipoles, then this
  vertex must belong to two distinct $(D-1)$-dipoles. Parallel edge
  count shows that this is not possible for $D\geq4$. 
  
  As illustrated by the right hand side of
  Fig.~\ref{fig:nondisjointchains}, for $D=3$ a vertex $u_{\circ}$ can belong to two
  2-dipoles $u_{\circ}-v_{\bullet}$ and $u_{\circ}-w_{\bullet}$. None of the edges incident at $u_{\circ}$ can be the root edge $r(G)$.
  If $u_{\circ}-v_{\bullet}$ belongs to a proper chain,
  then $w_{\bullet}$ has to belong to the same chain (since the chain has at
  least 4 internal vertices), hence there exists a vertex $w'_{\circ}$ which is connected to $w_{\bullet}$ by a pair of non-root edges
  and to $u_{\circ}$ by at least one non-root edge. 
   
   Applying the same reasoning to the $u_{\circ}-w_{\bullet}$ dipole, we conclude that the graph reduces to a double cycle
  of length 4 (on the right in Fig.~\ref{fig:nondisjointchains}), and none of the edges can be the root of $G$, which is impossible.

 \begin{figure}[ht]
\begin{center}
\psfrag{v}{$v_{\bullet}$}
\psfrag{u}{$u_{\circ}$}
\psfrag{w}{$w_{\bullet}$}
\psfrag{w'}{$w'_{\circ}$}
\includegraphics[scale=.5]{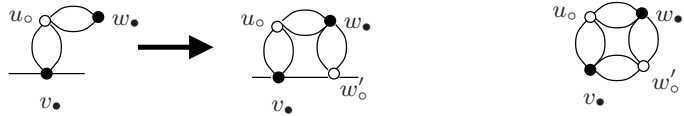}
\end{center}
\caption{A proper chain and a 2-dipole sharing vertices in $D=3$ and the double cycle graph of length 4.} \label{fig:nondisjointchains}
\end{figure}
  
  \qed
\end{proof}

Note that, as shown in Fig.~\ref{fig:nondisjointchains}, in $D=3$ a maximal proper chain can share vertices with a $(D-1)$-dipole (a 2-dipole), but this is not possible for $D \ge 4$.

\subsection{Classification of chains}
 
There are two main types of proper chains, depending on the way the half-edges are involved in the faces of $G$. 

Let us consider a proper chain with external colors $ ( c_1,c_2 )$, and let us denote the external colors of the first dipole from the left 
in the chain $ ( c_1,c' )$.
The two left half-edges $\ell_{\circ}$ and $\ell_{\bullet}$ belong to the same face with colors $ (c_1,c)$ for all $c\neq c'$.
The last face, with colors $ ( c_1,c' )$, travels horizontally to the next dipole. Iterating we are in one of the two cases below:
\begin{itemize}
 \item either the face $ ( c_1,c' )$ \emph{does not travel horizontally} through the entire chain, that is there exists a dipole in the chain
  such that its right external color is neither $c_1$ nor $c'$.
 \item or the face $ (c_1,c' ) $ \emph{travels horizontally} along the chain all the way to the right half-edges.
\end{itemize}
Furthermore the chains can have an even or odd number of dipoles, hence the chains are in one of the two cases below:
\begin{itemize}
 \item either the chain has an odd number of dipoles, hence $l_{\circ}$ and $r_{\circ}$ are both on the top of the chain, and $l_{\bullet}$
    and $r_{\bullet}$ are both on the bottom.
 \item or the chain has an even number of dipoles, hence $l_{\circ}$ and $r_{\bullet}$ are both on the top of the chain, and $l_{\bullet}$ 
    and $r_{\circ}$  are both on the bottom.
\end{itemize}

Correspondingly the chains are classified into:
\begin{description}
\item[\it Broken chains] A proper chain with external colors $(c_1,c_2)$ is
  \emph{broken} (or a $B$ chain) if for all $c\neq c_1$, $\ell_\circ$ and $\ell_\bullet$
  are in the same face of color $\{c,c_1\}$ and consequently $r_\circ$
  and $r_\bullet$ are in the same face of color $\{c,c_2\}$ for all $c\neq c_2$. 
  
  Broken chains are subdivided further into chains with equal external colors $c_2=c_1$ ($ B_{=}$ chains)
  and chains with distinct external colors, $c_2\neq c_1$ ($B_{\neq}$ chains). Furthermore, for each of the two cases 
  the chains can have an even ($B_{ = ; \genfrac{}{}{0pt}{}{\circ \bullet}{\bullet \circ} }$ resp. $B_{ \neq ; \genfrac{}{}{0pt}{}{\circ \bullet}{\bullet \circ} }$ chains) or an odd
  ($B_{ = ; \genfrac{}{}{0pt}{}{\circ \circ}{\bullet \bullet} }$ resp. $B_{ \neq ; \genfrac{}{}{0pt}{}{\circ \circ}{\bullet \bullet} }$ chains)
  number of dipoles.
\item[\it Unbroken chains] Chains that are not broken are \emph{unbroken} ($U$ chains). Let us
  consider separately chains of external colors $(c_1, c_2 \neq c_1)$ and
  $(c_1,c_1)$:
  \begin{itemize} 
  \item external colors $(c_1, c_2\neq c_1)$: $\ell_\circ$ has color $c_1$ and
    belongs to a face of color $\{c_1,c_2\})$ which does not contain
    $\ell_\bullet$. This face travels horizontally and 
    leaves the chain through $r_\circ$ (which has color $c_2$). The chain
    has an odd number of dipoles ($U_{\neq, \genfrac{}{}{0pt}{}{\circ \circ}{\bullet \bullet}}$ chain).  
  \item external colors $(c_1,c_1)$: again there is a face containing
    $\ell_\circ$ and not $\ell_\bullet$ (since the chain is not
    broken), which has to travel horizontally and leave the chain via
    $r_\bullet$. There is only one such face traveling
    horizontally between $\ell_\circ$ and $r_\bullet$, with color
    $\{c_1,c'\}$ (and a parallel face of color $\{c_1,c'\}$ goes through
    $\ell_\bullet$ and $r_\circ$). The color $c'$ is referred to as the
    \emph{secondary color} of the unbroken chain.  The chain has an even number of dipoles ($U_{ = , \genfrac{}{}{0pt}{}{\circ \bullet}{\bullet \circ}}$ chain). 
  \end{itemize}
\end{description}

We call \emph{internal faces} of a chain the faces involving only the internal edges of the chain, and \emph{external faces} of the chain the 
faces involving the half-edges.

\begin{lemma}\label{lem:chainfacescount}
 The numbers $F^{\rm int}(B)$ of internal faces of a broken proper chain $B$ with $2k(B)$ vertices and $F^{\rm int}(U)$ of internal faces of an unbroken proper chain $U$ with $2k(U)$ vertices are respectively:
 \[
   F^{\rm int}(B) = \frac{D(D-1)}{2} k(B) - D \;, \qquad 
   F^{\rm int}(U) = \frac{D(D-1)}{2} k(U) - D+1 \;.
 \]
 \end{lemma}
\begin{proof}
 By connecting the left (resp. the right) external half-edges of the chain $B$ (resp. $U$) into edges $\ell= \Braket{\ell_{\circ},\ell_{\bullet}}$ and $r=\Braket{r_{\circ},r_{\bullet}}$
 and marking $\ell$ as root, the chain becomes a melonic graph $G_{B}$ (resp. $G_{U}$) with $2k(B)$ (resp. $2k(U)$) vertices and 
 $F(G_{B}) = F^{\rm int}(B)  + 2D $ (resp.  $F(G_{U}) = F^{\rm int}(U)  + 2D -1$ faces.
 
 \qed 
\end{proof}

We associate to every kind of chain a \emph{chain-vertex} (see Fig.~\ref{fig:chains}). The chain-vertices have four half-edges, each incident to a black or white 
square (tracking the vertices of the first respectively last dipole in the chain). For broken chains the top and bottom squares are connected 
by all the external faces, while for the unbroken ones one external face travels horizontally. 

\begin{figure}[ht]
\begin{center}
\psfrag{c}{$c_1$}
\psfrag{c'}{$c_2\neq c_1$}
\psfrag{c''}{$c'\neq c_1$}
\psfrag{Bee}{$B_{ = ; \genfrac{}{}{0pt}{}{\circ \bullet}{\bullet \circ} }$}
\psfrag{Bne}{$B_{ \neq ; \genfrac{}{}{0pt}{}{\circ \bullet}{\bullet \circ} }$ }
\psfrag{Beo}{$B_{ = ; \genfrac{}{}{0pt}{}{\circ \circ}{\bullet \bullet} }$}
\psfrag{Bno}{$B_{ \neq ; \genfrac{}{}{0pt}{}{\circ \circ}{\bullet \bullet} }$}
\psfrag{Uno}{$U_{\neq, \genfrac{}{}{0pt}{}{\circ \circ}{\bullet \bullet}}$}
\psfrag{Uee}{$U_{ = , \genfrac{}{}{0pt}{}{\circ \bullet}{\bullet \circ}}$}
\includegraphics[scale=0.5]{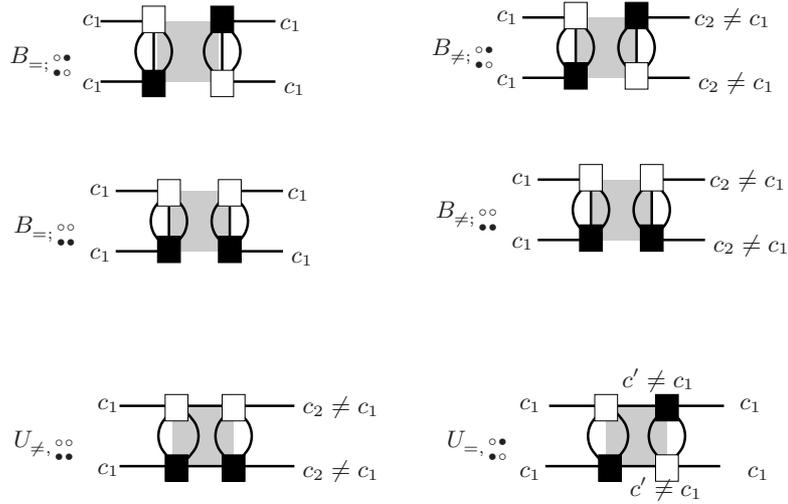}
\end{center}
\caption{The six chain-vertices for $D=3$. \label{fig:chains}}
\end{figure}

Each chain-vertex admits a minimal realization as a proper chain of $(D-1)$-dipoles:
$B_{ = ; \genfrac{}{}{0pt}{}{\circ \bullet}{\bullet \circ} }$ requires four dipoles,
 $B_{ \neq ; \genfrac{}{}{0pt}{}{\circ \bullet}{\bullet \circ} }$ requires two dipoles,
$B_{ = ; \genfrac{}{}{0pt}{}{\circ \circ}{\bullet \bullet} }$  requires three dipoles,
 $B_{ \neq ; \genfrac{}{}{0pt}{}{\circ \circ}{\bullet \bullet} }$ requires three dipoles,
$U_{\neq, \genfrac{}{}{0pt}{}{\circ \circ}{\bullet \bullet}}$ requires three dipoles,
$U_{ = , \genfrac{}{}{0pt}{}{\circ \bullet}{\bullet \circ}}$ requires two dipoles.

\newpage

\section{Schemes}\label{sec:schemes}
\subsection{Reduced schemes and the scheme of a colored graph}
\begin{definition}
A rooted \emph{scheme} is a connected rooted graph with colored edges having two types
of vertices:
\begin{itemize}
\item regular black and white vertices of degree $D+1$, each incident to one edge of each color,
\item chain-vertices of the 6 types 
$B_{ = ; \genfrac{}{}{0pt}{}{\circ \bullet}{\bullet \circ} }, B_{ \neq ; \genfrac{}{}{0pt}{}{\circ \bullet}{\bullet \circ} },B_{ = ; \genfrac{}{}{0pt}{}{\circ \circ}{\bullet \bullet} },
B_{ \neq ; \genfrac{}{}{0pt}{}{\circ \circ}{\bullet \bullet} }, U_{\neq, \genfrac{}{}{0pt}{}{\circ \circ}{\bullet \bullet}},U_{ = , \genfrac{}{}{0pt}{}{\circ \bullet}{\bullet \circ}}$,
having two white and two black squares each,
\end{itemize}
and edges connecting:
\begin{itemize}
 \item a black regular vertex and a white regular,
 \item a black (or white) regular vertex with a white (or black) square,
 \item a black and a white square.
\end{itemize}

A scheme is \emph{reduced} if:
\begin{itemize}
 \item it is melon free, 
 \item the left (and right) half-edges of any chain-vertex are not matched together into a non-root edge (they can however be matched in the root edge),
 \item the left (or right) half-edges of any chain-vertex or $(D-1)$-dipole are not matched both at the same time to the left (or right) half-edges of any  
 other chain-vertex or $(D-1)$-dipole.
\end{itemize}
\end{definition}

\begin{figure}[ht]
  \begin{center}
   \includegraphics[width=8cm]{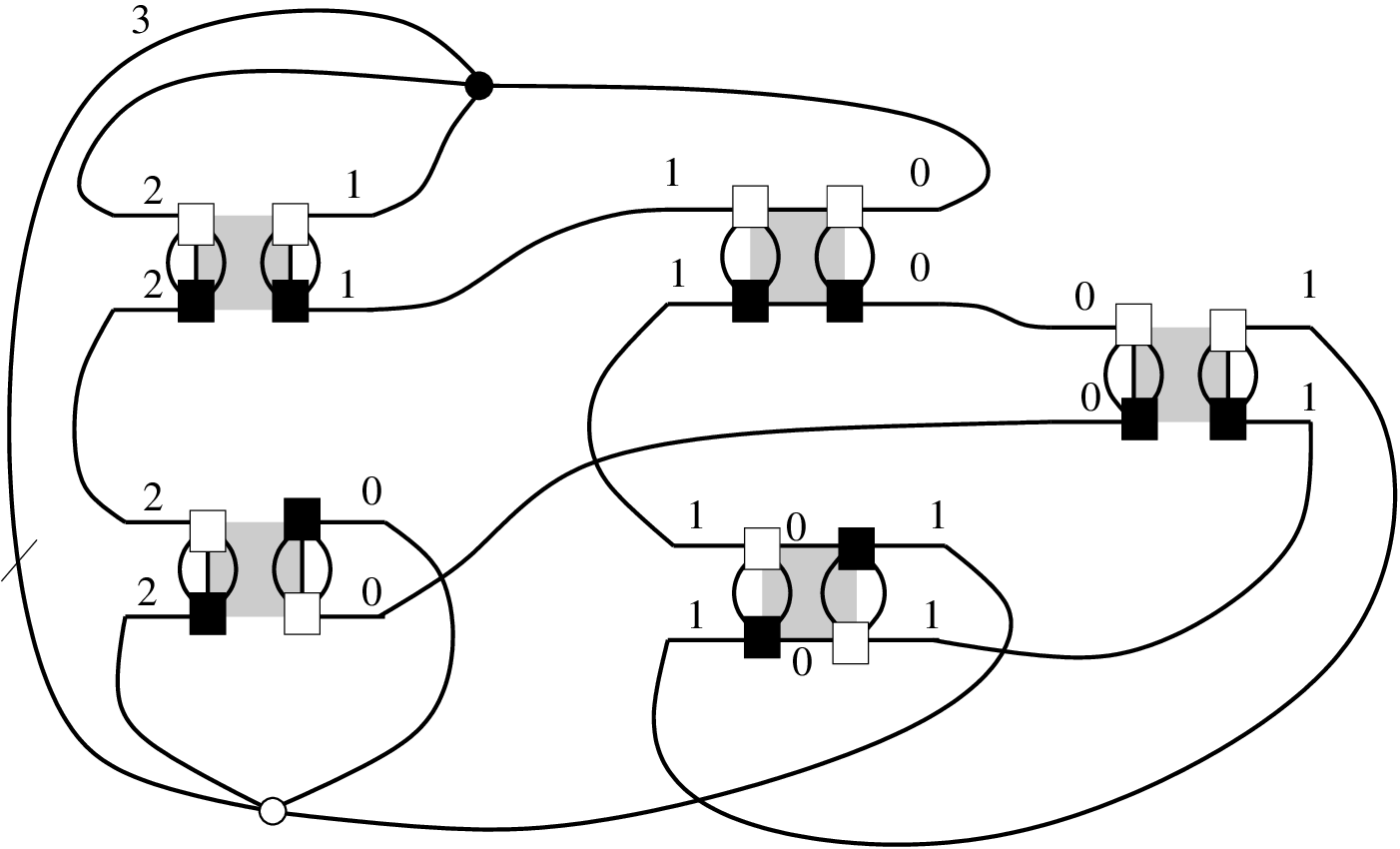}
    \caption{An example of scheme.}
    \label{fig:scheme}
  \end{center}
\end{figure}

\begin{definition}
The \emph{scheme $\tilde G$ of a melon-free colored graph $\hat G$} is obtained by
replacing each maximal proper chain of $\hat G$ by the corresponding
\emph{chain-vertex} (since maximal proper chains are vertex disjoint
this can be done independently for each chain). 
\end{definition}

By construction, the scheme of a melon-free colored graph is reduced. 

Observe that for $D=3$, replacing simultaneously the maximal chains of
a melon-free colored graph by chain-vertices is not always possible:
isolated $(D-1)$-dipoles are chains (although not proper chains), and
are not necessarily vertex disjoint. This is why we restrict our attention to
maximal \emph{proper} chains.

The following theorem is an immediate consequence of the previous discussion.
\begin{theorem}
\label{thm:scheme}
There is a bijection between the set of melon-free colored
graphs $\hat G$ with $2k(\hat G)$ vertices and the set of pairs $(\tilde G;
(C_1,\ldots,C_q))$ where $\tilde G$ is a reduced scheme with $q$
chain-vertices $x_1,\ldots,x_q$, and $C_i$ is a chain of the same type
as $x_i$, such that the total number of vertices in $\tilde G$ and the chains
$C_1,\ldots,C_q$ is $2k(\hat G)$. 
\end{theorem}

The chain-vertices we have introduced allow to keep track in $\tilde G$ of the faces of the melon-free 
graph that are not internal faces of some chain.
The following result is a direct consequence of Lemma \ref{lem:chainfacescount} and the definition of the degree, Eq.~\eqref{eq:defdegree}.

\begin{lemma}\label{lem:degsree-cheme} 
Let $\hat G_1$ and $\hat G_2$ be two melon-free rooted colored graphs
with the same scheme $\tilde G$. Then $\hat G_1$ and $\hat G_2$ have the
same degree.
\end{lemma}

\begin{definition}
The \emph{degree of a reduced scheme} is the common degree of all cores that have it as scheme. 
\end{definition}

\subsection{Schemes of fixed degree}

Our main interest in schemes is that unlike the set of cores of fixed
degrees, the set of reduced schemes of fixed degree is finite.

\begin{theorem}\label{thm:finiteness}
The number of reduced schemes with a fixed degree is finite.
\end{theorem}

\begin{proof} 
The proof of this is divided into two parts. First we analyze the iterative elimination
of chain-vertices, $(D-1)$-dipoles, and in the cases $D\ge 4$, $(D-2)$-dipoles in a reduced scheme, to prove the following result:
\begin{proposition}\label{prop:schemebound}
The number of chain-vertices, $(D-1)$, and for $D\geq 4$,
$(D-2)$-dipoles in a reduced scheme of degree $\delta$ is bounded by
$19\delta$.
\end{proposition}

\begin{proof} See Section~\ref{sec:prf1}.

\qed
\end{proof}

Once this result is granted, we observe that the minimal realization
of any chain-vertex consists of at most four $(D-1)$-dipoles, so that
there is an injective map from the reduced schemes of degree $\delta$ into melon free colored 
graphs of degree $\delta$ with at most $76\delta$ $(D-1)$- and $(D-2)$-dipoles. 

We then establish the following result:
\begin{proposition}\label{prop:finitdipole}
For $D=3$, the number of melon free colored graphs with fixed degree and a fixed number of
2-dipoles is finite.  For $D\geq4$, the number of melon free colored graphs with fixed degree and 
fixed numbers of $(D-1)$-dipoles and
$(D-2)$-dipoles is finite.
\end{proposition}

\begin{proof} See Section~\ref{sec:prf2}.

\qed
\end{proof}

Theorem~\ref{thm:finiteness} is an immediate consequence of these two propositions.

\qed
\end{proof}

\newpage

\section{Proof of Proposition~\ref{prop:schemebound}}\label{sec:prf1}

As a preliminary we analyze the effect that the deletion of a single
$(D-q)$-dipole has on the degree. We then extend the analysis to the deletion of a
chain-vertex. The conclusion of this analysis is that the deletion of a $(D-q)$-dipole:
\begin{itemize}
 \item either separates the graph into $q+1$ connected components (completely separating deletion), in which case the degree is distributed among the connected components,
 \item or it separates the graph into less than $q+1$ connected components, in which case the degree strictly decreases.
\end{itemize}
Similarly, for chains, the deletion:
\begin{itemize}
 \item either separates the graph into two connected components, in which case the degree is distributed among the connected components,
 \item or it does not separate the graph, in which case the degree strictly decreases.
\end{itemize}

\subsection{Analysis of a $(D-q)$-dipole removal}
 
Let us define more precisely the \emph{removal} of a $(D-q)$-dipole
(with $1\le q\le D-2$) in a colored graph $G$: assuming the half-edges
of the dipole have colors $c_0,c_1,\dots c_q $, we delete the two
vertices, the internal edges and the half-edges of the dipole and form one new
edge for each color $c_0,c_1,\dots c_q $ with the remaining pairs of
half-edges in $G$. 

By construction, the number of vertices decreases by 2 and
the number of edges decreases by $D+1$ at a deletion. In order to track the
variation of the degree we need to analyze more precisely the
variation of the number of faces. This is somewhat involved, as it depends on
the number of connected components the graph separates into upon
removal of the $(D-q)$-dipole and also on whether couples of new
edges belong to a same face or not.

\paragraph{\it Connected components and faces after a dipole removal.}
We denote:
\[G_1,G_2,\dots G_C \;, \]
the $C$ connected components obtained
after the removal of the $(D-q)$-dipole, $1\le C\le q+1$.  As the
removal of the dipole deletes two vertices we have 
\[ k(G) = k(G_1) + k(G_2) + \dots + k(G_C) + 1\; .  \] 

We denote $t_1$ the number of new edges belonging to $G_1$, $t_2$ the
number of new edges belonging to $G_2$, and so on. We have:
\[
 t_1+t_2 + \dots + t_C = q+1 \; .
\]
Without loss of generality we can assume that the colors of the $t_1$
new edges belonging to $G_1$ are $c_0,\dots c_{t_1-1}$, the colors of
the $t_2$ new edges belonging to $G_2$ are $c_{t_1},\dots
c_{t_1+t_2-1}$ and so on. 

The faces affected by the $(D-q)$-dipole removal are the ones
containing at least one of its vertices. They fall into three categories:
\begin{itemize}
 \item Faces with colors $\{c,c'\}$ such that $\{c,c'\} \cap
   \{c_0,c_1,\dots c_q \} = \emptyset$. For each of the
   $\binom{D-q}{2}$ choices of such pairs, exactly one face
   of degree two (an internal face of the dipole) is deleted
   by the removal of the dipole:
   \[
   F^{cc'}( G)=F^{cc'}( G_1)+F^{cc'}( G_2) + \dots + F^{cc'}( G_C) + 1 \; , 
   \] 
\item Faces with colors $\{ c_i ,c \}$, with $c_i\in \{ c_0,c_1,\dots c_q\}$
  and $c \notin \{ c_0,c_1,\dots c_q\} $. For each of the
  $(D-q)(q+1)$ choices of such pairs, exactly one face is incident
  to the dipole: this face has length at least four in $G$ and the dipole
  removal reduces its length by 2, so that the number of faces with these
  colors is unchanged:
  \[
  F^{c_ic}(G)=F^{c_ic}(G_1)+F^{c_ic}(G_2)+\dots + F^{c_ic}(G_C) \; ,
  \]
\item Faces with colors $\{ c_i, c_j \}$ with $\{c_i,c_j\} \subset
  \{c_0,c_1,\dots c_q \}$. For each of the $\binom{q+1}{2}$ choices
  of such colors, either one or two faces are incident to the dipole.
  In this case we distinguish between two possibilities:
     \begin{itemize}
      \item {\bf Type $a$.} (Fig.~\ref{fig:facesplit}) The four edges
        of color $c_i$ and $c_j$ belong to the same face of color $\{c_i,c_j\}$
        in $G$. Upon removal of the $(D-q)$-dipole this face of color
        $\{c_i,c_j\}$ splits into two disjoint faces:
              \[
               F^{c_ic_j}(G) = F^{c_ic_j}(G_1)+F^{c_ic_j}(G_2)+\dots + F^{c_ic_j }(G_C) -1 \; .
              \]
        
       \begin{figure}[ht]
       \begin{center}
        \psfrag{c1}{$c_1$}
        \psfrag{c2}{$c_2$}
            \includegraphics[width=7cm]{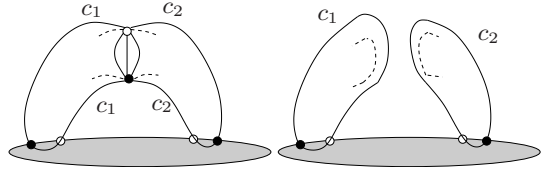}
       \end{center}
       \caption{A face which splits by deleting a dipole.\label{fig:facesplit}}
       \end{figure}
      \item {\bf Type $b$.} (Fig.~\ref{fig:facemerge}) The four edges of color $c_i$ and $c_j$ belong
        to two distinct faces $\{ c_i,c_j\}$ in $G$. Upon removal of the
        $(D-q)$-dipole the two faces $\{ c_i,c_j \}$ are merged into a unique
        face:
            \[
               F^{c_ic_j}(G) = F^{c_ic_j}(G_1)+F^{c_ic_j}(G_2)+\dots + F^{c_ic_j}(G_C) + 1 \; .
           \]
       \begin{figure}[ht]
       \begin{center}
        \psfrag{c1}{$c_1$}
        \psfrag{c2}{$c_2$}
            \includegraphics[width=7cm]{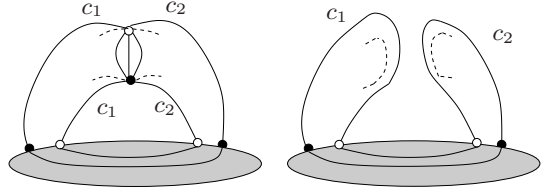}
       \end{center}
       \caption{Two faces which merge by deleting a dipole.\label{fig:facemerge}}
       \end{figure}
     \end{itemize}
\end{itemize}

A third possibility, namely that the two pairs of edges of color $\{c_i,c_j\}$ belong to the same face before the removal and the face 
does not split, does not exist because the graph $G$ is bipartite.  

We are now in position to describe the global effect of the removal of a dipole on the degree.

\paragraph{\it Case I. Completely separating $(D-q)$-dipoles.}
We first consider the case in which the removal of $(D-q)$-dipoles
splits the graph into $q+1$ connected components each containing
exactly one new edge ($C=q+1$ with the notation above): we refer to
such a dipole as \emph{completely separating}.  We illustrated this case in
Fig.~\ref{fig:separating}.

\begin{figure}[htb]
  \begin{center}
  \psfrag{c1}{$c_1$}
    \psfrag{c2}{$c_2$}
      \psfrag{c3}{$c_3$}
    \includegraphics[width=6cm]{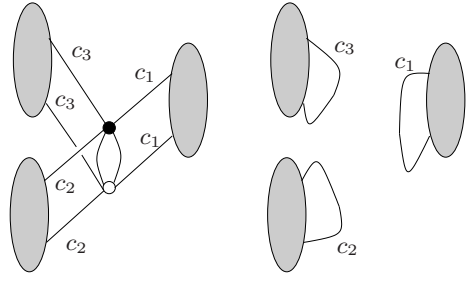}
    \caption{The decomposition at a completely separating $(D-2)$-dipole.} 
    \label{fig:separating}
  \end{center}
\end{figure}

In this case all the faces of colors $\{c_i,c_j\}$ with $\{c_i,c_j\} \subset
\{c_0,c_1,\dots c_q \} $ are of Type $a$, hence 
\begin{align*}F(G) & = F(G_1)
+ F(G_2) + \dots + F(G_{q+1}) + \binom{D-q}{2} - \binom{q+1}{2} \crcr 
&= F(G_1) + F(G_2) + \dots + F(G_{q+1}) +
\frac12{D (D-2q-1) } \; .  
\end{align*}
From Eq.~\eqref{eq:defdegree} we get: 
\begin{align*}
& \delta(G)   = \frac12D(D-1)k(G) + D - F(G) \\
            & \qquad = \frac12D(D-1) \Bigl( k(G_1) + k(G_2) + \dots + k(G_{q+1}) + 1 \Bigr) +D \\
            & \qquad \quad - \Big{(} F(G_1) + F(G_2) + \dots + F(G_{q+1}) + \frac12{D (D-2q-1) } \Big{)} \\
            & \qquad =\delta(G_1) +  \delta(G_2) + \dots  +  \delta(G_{q+1}) \; .
\end{align*}
 
 In conclusion the degree is distributed between the connected components created by the removal of a completely
 separating dipole.

\paragraph{\it Case II. Non completely separating $(D-q)$-dipoles.} 
We now consider the remaining cases ($1\le C\le q$): we refer to such
a dipole as \emph{non completely separating}.

All the faces $\{ c_i,c_j\} $ with $ c_i \in \{c_0,\dots c_{t_1-1} \} $ and $ c_j\notin
\{c_0,\dots c_{t_1-1} \} $ are of Type $a$. On the contrary the
faces $\{c_i,c_j\}$ with $\{c_i,c_j \} \subset \{c_0,\dots c_{t_1-1} \} $ can be
either of Type $a$ or of Type $b$.   
 
Let us denote $b_1$ the number of faces of Type $b$, $0\le b_1 \le
\binom{t_1}{2}$ in $G_1$, $b_2$ the number of faces of Type $b$
in $G_2$ and so on. We have:
\begin{align*} F( G)  = & F(G_1) + F(G_2) + \dots + F(G_{C}) + \binom{D-q}{2} - \binom{q+1}{2} \crcr
& \qquad + 2b_1 + 2b_2 + \dots +2b_C \crcr
= & F(G_1) + F(G_2) + \dots + F(G_{C}) + \frac12 D  (D-2q-1)\crcr
& \qquad + 2b_1 + 2b_2 + \dots +2b_C \;.
\end{align*}
Using again Eq.~\eqref{eq:defdegree} we get:
\begin{align*}
\delta(G)  =& \frac12 D(D-1)k(G) + D - F(G) \crcr
            =& \frac12D(D-1) \Bigl( k(G_1) + k(G_2) + \dots + k(G_{C}) + 1 \Bigr) +D \crcr
            & - \Big{(} F(G_1) + F(G_2) + \dots + F(G_{C})  \crcr
            & \qquad \qquad + \frac12{D (D-2q-1) }{} + 2b_1 + 2b_2 + \dots +2b_C  \Big{)} \crcr
            =& \delta(G_1) + \delta(G_2) + \dots  + \delta(G_{C}) \crcr
            & +D(q+1-C) - 2b_1 -2b_2-\dots -2b_C \; .
\end{align*}

In other terms the variation of the degree through the removal depends
on the structure of the incident faces. Observe that, as $b_i\le \binom{t_i}{2}$:
\begin{align}\label{eq:variationbydeletion}
& D(q+1-C) - 2b_1 -2b_2-\dots -2b_C \ge D(q+1-C) - t_1(t_1-1) -\dots t_C(t_C-1) \crcr
& = (D-t_1)(t_1-1) + \dots (D-t_C)(t_C-1) \;,
\end{align}
as $t_1+\dots +t_C =q+1$. For non completely separating deletions we have $t_i \ge 1$ and at least one of them (say $t_1$) 
is in fact at least $2$, hence 
\[
 (D-t_1)(t_1-1) + \dots (D-t_C)(t_C-1)  \ge D-t_1 \ge D- (q+1)\ge 1 
\]

We now make the various
possibilities more explicit in the cases of $(D-1)$- and
$(D-2)$-dipoles:
 \begin{description}
  \item [\it Non separating $(D-1)$-dipoles.] We have $q=1$, $C = 1$ and $t_1=2$. Depending on the value of $b_1$ we thus distinguish two
   cases:
   \begin{itemize}
    \item Case II.A: the face $(c_0,c_1)$ is of Type $a$, hence $b_1=0$ and:
     \[ \delta(G) = \delta(G_1)  + D \; .\]
    \item Case II.B: the face $(c_0,c_1)$ is of Type $b$, $b_1=1$ and:
    \[ \delta(G) = \delta(G_1)  + D -2 \; .\]
   \end{itemize}
  \item[\it Non completely separating $(D-2)$-dipoles.]  We have $q=2$ and there are two possible values for $C$, and a
   total of $6$ possible cases:
   \begin{itemize}
    \item  \emph{non separating $(D-2)$-dipole deletion}: $C=1$, $t_1 =3$, hence:
              \begin{itemize}
               \item $b_1 = 0$, $ \delta(G) = \delta(G_1)  + 2D $.
               \item $b_1 = 1$,  $ \delta(G) = \delta(G_1)  + 2D -2 $.  
               \item $b_1 = 2$, $ \delta(G) = \delta(G_1)  + 2D -4 $.  
               \item $b_1 = 3$,  $ \delta(G) = \delta(G_1)  + 2D -6 $. 
              \end{itemize}
    \item \emph{partially separating $(D-2)$-dipole deletion}: $C=2$, $t_1 = 1$ (hence $b_1 =0$), $t_2 = 2$, hence:
              \begin{itemize}
               \item $b_2 =0$, $ \delta(G) = \delta(G_1) + \delta(G_2)  + D $.
               \item $b_2 =1$,  $ \delta(G) = \delta(G_1) + \delta(G_2)  + D -2 $.
              \end{itemize}
   \end{itemize}   
\end{description}
   
   In conclusion, the degree decreases by at least $D-2$ by removal of a non separating $(D-1)$-dipole, and for $D\ge 4$ the degree decreases by at 
   least $D-2$ by removal of a non separating or partially separating $(D-2)$-dipole.
   
\subsection{Chain-vertex removal}

The removal of a chain-vertex consists in deleting this vertex and the
incident half-edges and creating two new edges by joining 
the half-edges that arise from the same extremity of the
chain-vertex.  The removal of a chain-vertex in a scheme $\tilde G$ can
equivalently be performed in the following way:
\begin{itemize}
\item replace the chain-vertex by its minimal length chain
  representation: this yields a scheme $\tilde G'$ with same degree.
\item remove one of the $(D-1)$-dipole of the inserted chain:
  one of the three cases above applies.
\item eliminate the melons that might have been created: these
  operations do not affect the degree.
\end{itemize}
This last procedure, although slightly more complex \emph{a priori},
allows to built on the case analysis already done for $(D-1)$-dipole
removal. 
 
\paragraph{\it Case I. Separating chain-vertex.}
The removal of a chain-vertex separates $\tilde G$ into two
components $\tilde G_1$ and $ \tilde G_2$ if and only if the deletion of any
$(D-1)$-dipole of the equivalent chain separates the graph $\tilde G'$ into
two components $\tilde G'_1$ and $\tilde G'_2$. In such a case,
\[
\delta(\tilde G)=\delta( \tilde G')=\delta( \tilde G'_1)+\delta( \tilde G'_2)=\delta(\tilde G_1)+\delta( \tilde G_2).
\]
Observe that in this case, the chain-vertex can represent
indifferently an unbroken or a broken chain.

\paragraph{\it Case II. Non-separating chain-vertex, broken chain.}

Let us call $(c_1,c')$ the external colors of the first dipole in the chain.
This case is similar to a Case II.A removal of a $(D-1)$-dipole: the removal
of the chain-vertex does not separate $\tilde G$ and in the
resulting scheme $\tilde G_1$ the two new edges belong to two different
$(c_1,c')$-cycles.  Then the removal of the chain-vertex is equivalent
to a Case II.A removal of $(D-1)$-dipole in the graph $\tilde G'$, followed by
some melon deletions:
\[ \delta(\tilde G) =\delta(\tilde G')  = \delta( \tilde G_1)  + D \; .\]

\paragraph{\it Case III. Non-separating chain-vertex, unbroken chain.} 
Let us call $(c_1,c_2)$ the external colors of the chain-vertex if they are different, 
or we  call the secondary color of the chain vertex $c_2$, if the external colors are $(c_1,c_1)$. We have two sub cases:
\begin{itemize}
\item {\it Case III.a: two resulting faces.}
This case is similar to a Case II.A removal of a $(D-1)$-dipole: the removal
of the chain-vertex does not separate $\tilde G$ and in the
resulting scheme $\tilde G_1$ the two new edges belong to two different
$(c_1,c_2)$-cycles.  Then the removal of the chain-vertex is equivalent
to a Case II.A removal of $(D-1)$-dipole in the graph $\tilde G'$, followed by
some melon deletions:
\[ \delta(\tilde G) =\delta(\tilde G')  = \delta( \tilde G_1)  + D \; .\]
 \item {\it Case III.b: single resulting face.}
 This case is similar to a Case II.B removal of a $(D-1)$-dipole: the removal
of the chain-vertex does not separate $\tilde G$ but in the
resulting scheme $ \tilde G_1$ the two new edges belong to a same $(c_1,c_2)$-cycle.
The removal of the chain-vertex is equivalent to a Case II.B removal of $(D-1)$-dipole
in the graph $ \tilde G'$, followed by some melon deletions:
\[\delta( \tilde G) =\delta( \tilde G')  = \delta( \tilde G_1)  + D -2 \; .\]

\end{itemize}

\subsection{Iterative deletion of chain-vertices, $(D-1)$-dipoles and $(D-2)$-dipoles}\label{sec:iterate}

We now present an algorithm which allows us to eliminate one by one chain vertices and
$(D-1)$-dipoles (and, for $D\ge 4$, $(D-2)$-dipoles) in a reduced scheme. This will allow us to show that 
the total number of chain vertices, $(D-1)$-dipoles and $(D-2)$-dipoles in a reduced scheme
is bounded linearly in terms of the degree. This algorithm is not unique: 
we present here an adaptation of a similar one introduced in \cite{Fusy:2014rba}. For $D=3$, the algorithm goes through ignoring the $(D-2)$-dipoles.

The deletions of chain vertices and $(D-1)$-dipoles can either be
completely separating or not separating. The deletions of $(D-2)$-dipoles can be
completely separating, partially separating or non separating. Under
any completely separating deletion the degree is distributed among the
resulting connected components, while under the non completely
separating deletions the degree decreases by at least $D-2$.

\subsubsection{Non completely separating deletions} 

We will first perform a maximal number of \emph{non completely separating deletions}, that is 
non separating deletions of chain vertices and $(D-1)$-dipoles and 
non separating or only partially separating deletions of $(D-2)$-dipoles. 
Some examples of the algorithm we present below are depicted in Fig.~\ref{fig:nonsepdeletions}
and Fig.~\ref{fig:nonsepdeletions1}.  

Let $\tilde G$ be a reduced scheme of degree $\delta(\tilde G)$, and let us denote $\mathfrak{D}(\tilde G)$ the total number of 
chain vertices, $(D-1)$-dipoles and $(D-2)$-dipoles in $\tilde G$. 
As long we find a non separating chain-vertex, a non separating $(D-1)$-dipole \emph{formed by parallel unmarked edges} 
or a non completely separating $(D-2)$-dipole \emph{formed by parallel unmarked edges} we delete it. 
We mark the new edges with a \emph{blue mark, keeping track of the multiplicity}. 
That is, if the two half-edges connected at a step come from edges having $m_{1;{\rm blue}}$ and $m_{2;{\rm blue}}$ blue marks respectively, 
the new edge will have $m_{1;{\rm blue}}+m_{2;{\rm blue}}+1$ blue marks.
Observe that the deletion of a chain-vertex having a left and a right half-edge matched into an edge creates only one edge with two blue marks. 
As we delete partially separating $(D-2)$-dipoles the scheme can become disconnected and one can 
obtain ring components with blue marks, see Fig.~\ref{fig:nonsepdeletions} and Fig.~\ref{fig:nonsepdeletions1}.

Our aim is to delete at each step only dipoles which were present in the  original scheme, hence we only delete dipoles with unmarked parallel edges.
If the deletion creates new $(D-1)$-dipoles and $(D-2)$-dipoles having blue marks on some of their parallel edges, these edges 
\emph{do not count} when identifying dipoles. However, the ring components consisting in one edge with blue marks \emph{count} as connected components when 
deciding whether a deletion is completely separating or not.

Observe that $\mathfrak{D}(\tilde G)$ always goes down by one under a chain-vertex deletion, but, for $3\le D \le 5$ it can go down by as much as 
$3$ for the dipole deletions (see Fig.~\ref{fig:nonsepdeletions}): 
\begin{itemize}
 \item in $D=3$ the $2$-dipoles (i.e. $(D-1)$-dipoles) are not
   necessarily vertex disjoint. Each vertex of a $2$-dipole can belong
   to another $2$-dipole, in which case three dipoles are erased by a
   deletion.
 \item in $D=4$ the $3$-dipoles and $2$-dipoles (i.e. $(D-1)$-dipoles
   and $(D-2)$-dipoles) and couples of $2$-dipoles
   (i.e. $(D-2)$-dipoles) are not necessarily vertex disjoint and
   again up to three dipoles can be erased at one step.
 \item in $D=5$ the $3$-dipoles (i.e. $(D-2)$-dipoles) are not
   necessarily vertex disjoint and again up to three dipoles can be
   erased at one step.
\end{itemize}

Observe that if two dipoles share a vertex, then neither of the two can be completely separating.

\begin{figure}[ht] 
\begin{center}
\includegraphics[width=6cm]{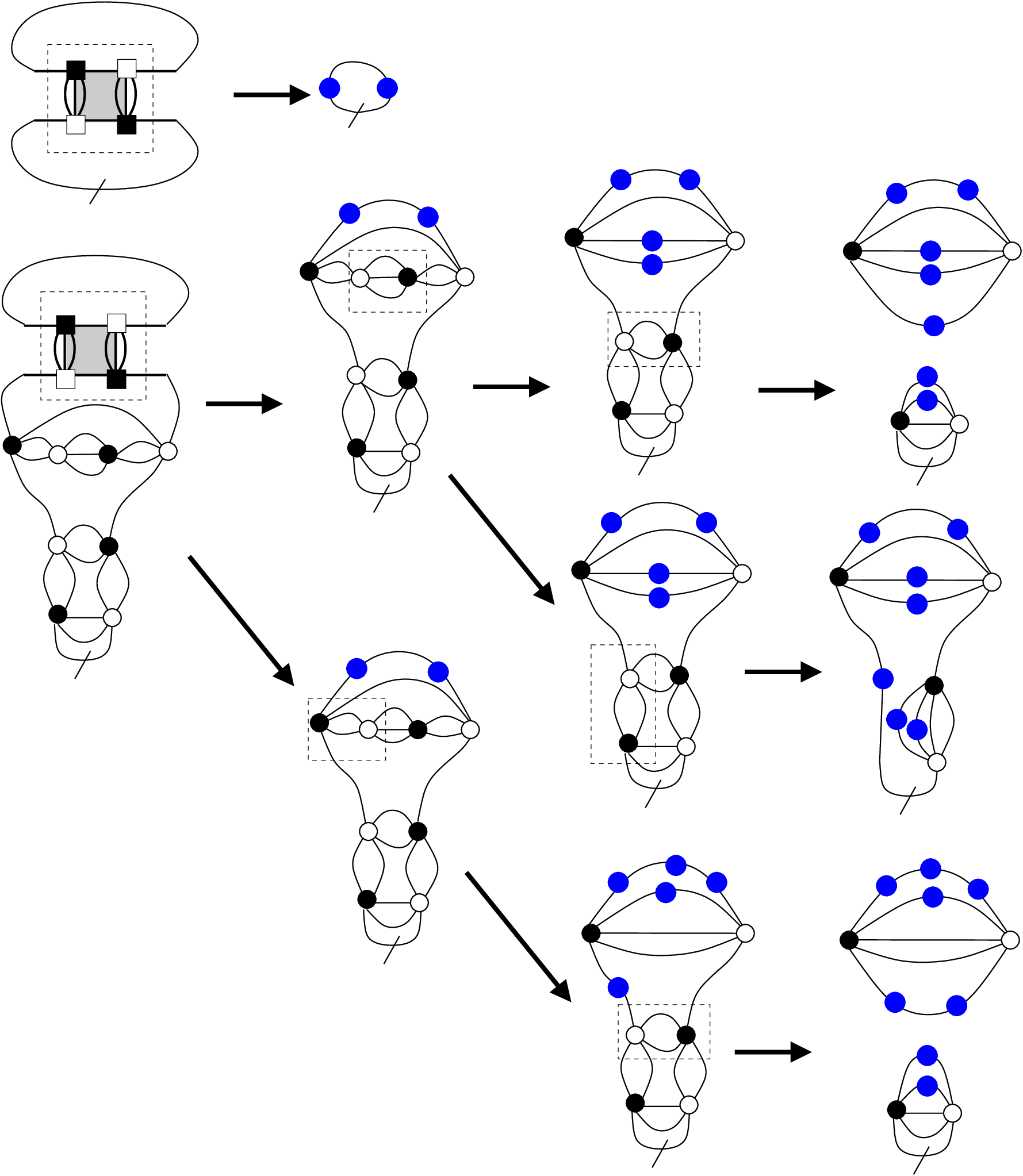}
\end{center}
\caption{Iterated non completely separating deletions, first two
  examples: in the first case only one deletion is necessary; in the
  second example 3 possible maximal sequences of deletions are shown.\label{fig:nonsepdeletions}}
\end{figure}

\begin{figure}[ht]
\begin{center}
\includegraphics[width=6cm]{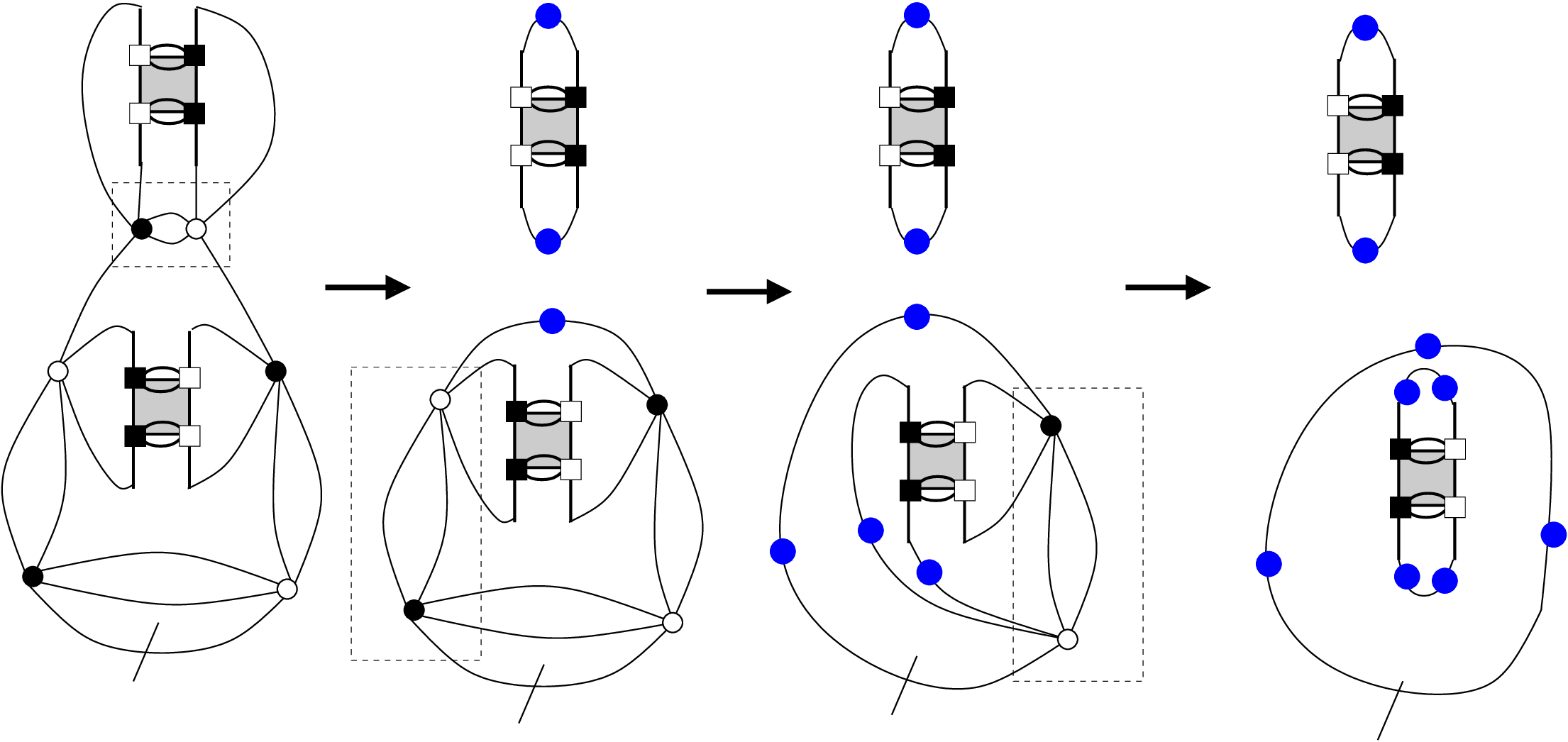}
\end{center}
\caption{Iterated non completely separating deletions, third example.\label{fig:nonsepdeletions1}}
\end{figure}

We iterate the non completely separating deletions maximally and obtain a scheme $\tilde G'$. Unlike $\tilde G$, $\tilde G'$ in general is neither connected nor reduced
(as depicted for instance in Fig.~\ref{fig:nonsepdeletions} and Fig.~\ref{fig:nonsepdeletions1}). Furthermore, $\tilde G'$ is not unique, and depends of the sequence of deletions performed.

The degree of $\tilde G'$ is the sum of the degrees of its connected components.
Let us denote the maximal number of non completely separating deletions one can perform starting from a reduced scheme by $q_{\rm{n.c.s.}}$.
We have the following inequalities:
\begin{itemize}
 \item the degree goes down by at least one for each of these deletions, hence:
\[
 q_{\rm{n.c.s.}} \le \delta ( \tilde G') + q_{\rm{n.c.s.}} \le \delta(\tilde G) \;,
\]
\item $\mathfrak{D}(\tilde G)$ goes down by at most three at each step, hence:
\[
 \mathfrak{D}(\tilde G) \le \mathfrak{D}(\tilde G') + 3 q_{\rm{n.c.s.}}  \le  \mathfrak{D}(\tilde G') + 3 \delta(\tilde G) \; ,
\]
\item every deletion creates at most three new blue marks, hence the 
total number of marks $\tilde G'$, $m_{\rm{blue}}(\tilde G')$, is bounded by:
\[
 m_{\rm blue}(\tilde G') \le 3 q_{\rm{n.c.s.}} \le 3  \delta(\tilde G) \;.
\]
\end{itemize}

In the scheme $\tilde G'$, all the chain-vertices, $(D-1)$ and $(D-2)$-dipoles are completely separating.
It follows that, for any $D$, all the remaining $(D-1)$ and $(D-2)$-dipoles \emph{are vertex disjoint}.

In order to conclude that $ \mathfrak{D}(\tilde G) $ is bounded linearly in terms of $ \delta(\tilde G)$,
it is enough to bound $\mathfrak{D}(\tilde G') $ in terms of $ \delta( \tilde G')$ and $ m_{\rm{blue}}(\tilde G')$. This is slightly subtle because the remaining deletions are all 
separating, hence they conserve the degree. 

All the remaining $(D-1)$- and $(D-2)$-dipoles are vertex disjoint, hence we can delete them together with the remaining chain-vertices in any order. 
We mark the new edges with a black mark, \emph{keeping track of the multiplicity} (and of course of the multiplicity of the blue marks). 
That is, if the two half-edges connected at a step come from edges having $m_{1;{\rm black}}$ and $m_{1;{\rm blue}} $ respectively
$m_{2;{\rm black}}$ and $m_{2;{\rm blue}} $ black and blue marks, the new edge will have 
$m_{1;{\rm black}}+ m_{2;{\rm black}} + 1 $ black marks and $ m_{1;{\rm blue}} + m_{2;{\rm blue}}$ blue marks.

For each $(D-2)$-dipole deleted (having external colors, say, $c_1,c_2$ and $c_3$) 
we add a copy of the fundamental melon (consisting in two vertices connected by $D+1$ edges) and add a black mark on its edges $c_1,c_2$ and $c_3$. 
Of course, edges with either type of marks (black or blue) do not count when identifying $(D-2)$ and $(D-1)$-dipoles.

As before, these deletions can  create ring components with marks.
We represent in Fig.~\ref{fig:nonsepdeletions2} the deletion of a maximal set of non separating dipoles and chain vertices, followed
by the deletion of all the separating $(D-1)$ and $(D-2)$-dipoles and chain vertices in a reduced scheme.
\begin{figure}[ht]
\begin{center}
\includegraphics[width=10cm]{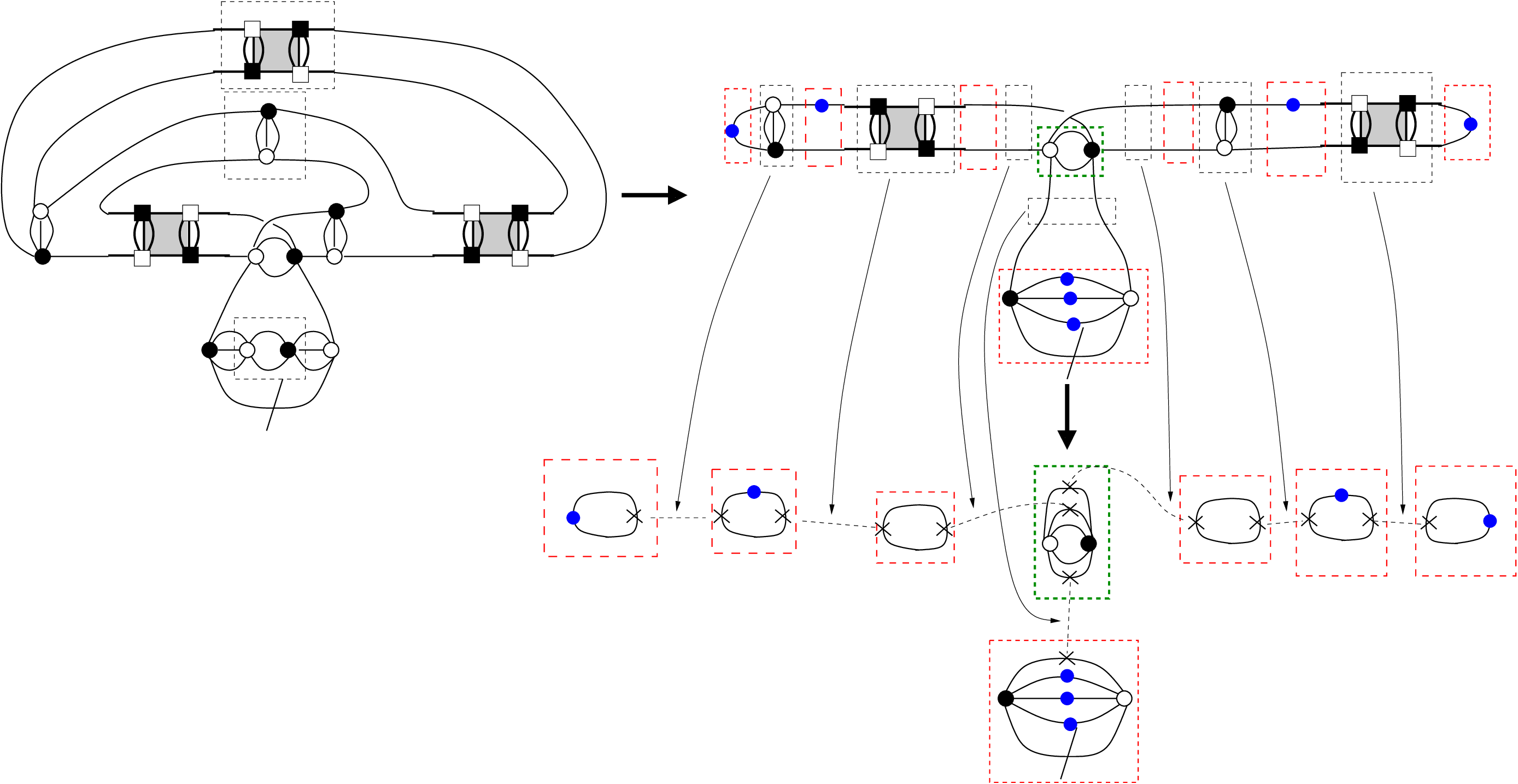}
\end{center}
\caption{Maximal deletion of non separating chain vertices and dipoles, followed by the deletion of the separating chain vertices and dipoles in a reduced scheme.\label{fig:nonsepdeletions2}}
\end{figure}

Let us denote the scheme obtained in this way by $\tilde G''$. By construction $\tilde G''$ has:
\begin{itemize}
 \item blue and black marks. The number of blue marks of $\tilde G''$ is equal to the number of blue marks 
 of $\tilde G'$, $m_{\rm blue}(\tilde G'') = m_{\rm blue}(\tilde G')$. We denote $m_{\rm black}(\tilde G'') $ the number of black marks in $\tilde G''$.
 \item the same degree as $\tilde G'$.
 \item ring components with marks (either blue or black).
 \item no $(D-1)$-dipole made of unmarked parallel edges.
 \item no chain-vertex.
 \item all the $(D-2)$-dipoles made of unmarked parallel edges are copies of the fundamental melon with three marked edges.
 \item at least a mark (blue or black) in every connected component (or no mark and only one connected component if no dipole or chain-vertex is ever deleted).
\end{itemize}

The connected components of $\tilde G''$ can be seen as the vertices of an abstract graph $\mathfrak{F}$ (represented in Fig.~\ref{fig:nonsepdeletions2} at the bottom) 
whose edges correspond either to the chain-vertices and to the $(D-1)$-dipoles in $\tilde G'$, or to pairs of half-edges of the same color of the $(D-2)$-dipoles in $\tilde G'$.
A subtle point is the following (see Fig.~\ref{fig:nonsepdeletions2}): a pair of half-edges of the same color of 
a $(D-2)$-dipole in $\tilde G'$ can lead:
\begin{itemize}
 \item either \emph{to an edge and a vertex} in $\mathfrak{F}$ if the pair is matched to a pair of half-edges of a chain-vertex, $(D-1)$- or $(D-2)$-dipole, (this is the case of the left and right pairs around the middle dipole in Fig.~\ref{fig:nonsepdeletions2})
 \item or \emph{to just an edge} in $\mathfrak{F}$ if it is not (this is the case of the bottom pair around the middle dipole in Fig.~\ref{fig:nonsepdeletions2}).
\end{itemize}

As all the chain vertices and $(D-1)$- and $(D-2)$-dipoles of $\tilde G'$ are completely separating, $\mathfrak{F}$ is a \emph{forest} 
: every tree in $\mathfrak{F}$ corresponds to one of the connected components of $\tilde G'$. 
We denote $\mathfrak{D}^{ (D-1)}$ the number of chain vertices and $(D-1)$-dipoles of $\tilde G'$ and 
$ \mathfrak{D}^{   (D-2)} $ the number of $(D-2)$-dipoles of $\tilde G'$, hence:
\[
 \mathfrak{D}(\tilde G')  = \mathfrak{D}^{ (D-1)} + \mathfrak{D}^{   (D-2)} \; , \qquad  \mathfrak{D}^{  (D-1)} + 3 \mathfrak{D}^{   (D-2)}   =  E(\mathfrak{F})\; ,
\]
where $E(\mathfrak{F})$ denotes the number of edges of the abstract forest $\mathfrak{F}$,

The following lemma characterizes the components of degree zero having only black marks in $\tilde G''$.

\begin{lemma}\label{lem:zerodegcomp}
 The components of degree zero having only black marks of $\tilde G''$ and not containing the root:
 \begin{itemize}
  \item either are rings with only two marks such that at least one of the marks comes from a $(D-2)$-dipole deletion,
  \item or have at least three marks.
 \end{itemize}
\end{lemma}

\begin{proof} The crucial observation is that all the unmarked edges in $\tilde G''$ are in fact edges which were present in $\tilde G$, and 
$\tilde G$ is a reduced scheme. 

The components of degree zero of $\tilde G''$ are either ring components or melonic graphs with at least two vertices.

Ring components with one black mark can be created by separating deletions only if one deletes $D-1$ edges in a melon $\mathbb{O}^c$ in $\tilde G$,
one deletes a chain-vertex whose left (or right) external half-edges are matched together into an edge in $\tilde G$  or one deletes $D-2$
edges in a $(D-1)$-dipole of $G$. The first two cases are impossible as $\tilde G$ is reduced, while the third would mean that a dipole of $\tilde G$
has been mislabeled.

Ring components with two marks can be created by the deletion of only $(D-1)$-dipoles and chain vertices only if, in $\tilde G$, the two left (or right) half-edges of 
a $(D-1)$-dipole or a chain-vertex are both matched to the two left (or right) half-edges of another $(D-1)$-dipole or a chain-vertex.
This is again impossible, as $\tilde G$ is reduced.

Ring components with three marks can be created, as depicted in Fig.~\ref{fig:ring} below.
\begin{figure}[ht]
\begin{center}
\psfrag{c}{$c$}
\includegraphics[width=7cm]{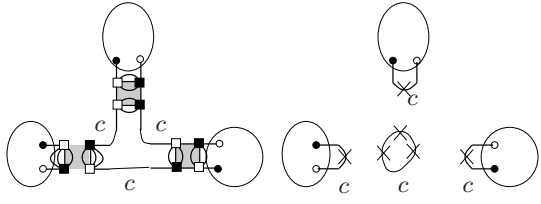}
\end{center}
\caption{Ring components with three marks. \label{fig:ring}}
\end{figure}

A nontrivial melonic graph with fewer than three marks contains at least a $(D-1)$-dipole made of unmarked parallel edges, which is impossible
by the construction of $\tilde G''$.

\qed
\end{proof}

The remainder of the proof relies on the observation that $\mathfrak{F}$ is a finite forest
as it has a finite number of vertices of valence zero or one (as only the root component, the strictly positive 
degree components, and the zero degree components with blue marks of $\tilde G''$ can be 
zero valent or univalent in $\mathfrak{F}$) and a finite number of vertices of valence two (as this number is  bounded linearly in the number 
of $(D-2)$-dipoles in $\tilde G'$).
We denote:
\begin{itemize}
 \item $n_+\le \delta(\tilde G')$ the number of connected components of $\tilde G''$ having strictly positive degree.
 \item $n_{0,{\rm blue}}\le m_{\rm blue}(\tilde G')$ the number of connected components of $\tilde G''$ having degree zero and having at least a blue mark.
 \item $n_{0,{\rm black}}^{(2)} \le  3 \mathfrak{D}^{   (D-2) } $ the number of connected components of degree zero of $\tilde G''$ having exactly two black marks
 and no blue mark. They are necessarily ring components and at leas one of the marks comes from a $(D-2)$-dipole deletion.
 \item $n_{0,{\rm black}}^{(3)}$ the  number of connected components of degree zero of $\tilde G''$ having only black marks and having at least three marks.
\end{itemize}
 We have:
\begin{align*}
 1+ n_+ + n_{0,{\rm blue}} + n_{0,{\rm black}}^{(2)} + n_{0,{\rm black}}^{(3)} & \ge  E(\mathfrak{F}) +1 \; , \crcr
  2 n_{0,{\rm black}}^{(2)} + 3 n_{0,{\rm black}}^{(3)} & \le  m_{\rm black}(\tilde G'') = 2E(\mathfrak{F}) \; , 
\end{align*}
that is $n_{0,{\rm black}}^{(3)}   \le 2 n_+ + 2 n_{0,{\rm blue}} $,
which further leads to:
\begin{align*}
 \mathfrak{D}^{  (D-1)} + 3 \mathfrak{D}^{   (D-2)}  &= E(\mathfrak{F})  \le 3 n_+ + 3 n_{0,{\rm blue}} + n_{0,{\rm black}}^{(2)}  \Rightarrow\crcr
  \mathfrak{D}^{  (D-1)} + 3 \mathfrak{D}^{   (D-2)}  & \le   3 n_+ + 3 n_{0,{\rm blue}}    + 3 \mathfrak{D}^{   (D-2)}    \Rightarrow 
 \mathfrak{D}^{  (D-1)} \le 3 n_+ +   3  n_{0,{\rm blue}}  \; .
\end{align*}
On the other hand, as only the root component, the positive degree components or the zero degree components with blue marks can be univalent in $\mathfrak{F}$, 
$\mathfrak{D}^{(D-2)} $ is bounded by the maximal number of trivalent vertices in a forest with exactly 
$ 1 + n_+ + n_{0,{\rm blue}}$ univalent vertices, 
$ \mathfrak{D}^{(D-2)} \le n_+ + n_{0,{\rm blue}} - 1 \;, $
\ie the number of such vertices in a binary tree. Thus:
\[
 \mathfrak{D}(\tilde G')  \le 4 n_+ +  4 n_{0,{\rm blue}}  - 1  \le 4  \delta(\tilde G' ) + 4 m_{\rm blue}(\tilde G') \;,
\]
which finally leads us to:
\[
 \mathfrak{D}(\tilde G) \le  3 \delta(\tilde G) +  \mathfrak{D}(\tilde G')  \le  3 \delta(\tilde G) + 4  \delta(\tilde G' ) + 4 m_{\rm blue}(\tilde G') \le 
 19 \delta( \tilde G) \;,
\]
as $\delta(\tilde G') \le \delta(\tilde G)$ and $ m_{\rm blue}(\tilde G')  \le 3  \delta(\tilde G) $.

This concludes the proof of Proposition~\ref{prop:schemebound}.

\newpage

\section{Proof of Proposition~\ref{prop:finitdipole}}\label{sec:prf2}
 
\subsubsection{The case $D=3$.}

For $D=3$, we are interested in melon free graphs with fixed number of 2-dipoles,
or equivalently, with a fixed number of faces of degree 2. In view of
Eq.~\eqref{eq:delta'}, the number of faces of degree 6 or more in
such a graph satisfies:
\[
\sum_{p\geq3}F_p( \hat G)\leq  \sum_{p\geq2}2(p-2)F_p(\hat G) \leq (D+1)\delta (\hat G)+ 2 F_1(\hat G) \; ,
\]
i.e. this number is finite. Moreover the number of vertices
incident to a face of degree 6 or more is finite:
\[
\sum_{s\geq3}2pF_p (\hat G)=\sum_{p\geq3} 2 [ p-2+2]F_p (\hat G) \leq 5(D+1)\delta (\hat G) +10 F_1 (\hat G) .
\]
The vertices 
belonging to a face of degree not equal to $4$ either belong to a face of degree 
two or to a face of degree larger than or equal to $6$. According to 
our previous remark, the number of such vertices is at most:
\[5(D+1)\delta (\hat G) +12F_1 (\hat G) \; . \]

However there could \emph{a priori} be an arbitrary number of faces of
degree $4$ (that is arbitrarily many vertices incident only to faces
of degree $4$), since the coefficient of $F_2$ in
Eq.~\eqref{eq:delta'} is zero for $D=3$. Let us rule this possibility
out.

Let us count the maximal number of vertices that can be at distance at
most $3$ of a vertex on a face of degree not equal to $4$. 
Since all vertices have degree $4$, the number of
vertices at distance at most 3 of any vertex is $4+4\cdot 3+4\cdot 3^2=52$.
Therefore if a colored graph $G$ has more than
$[5(D+1)\delta (\hat G) +12F_1 (\hat G) ]52$ vertices, then it contains a vertex $v$ such that
all vertices at distance less than 3 of $v$ belong \emph{only} to faces of
length $4$.

Now take an arbitrary jacket of $\hat G$: for instance the one corresponding to the cycle
$(0,1,2,3)$. Then the faces of color $(0,1)$, $(1,2)$, $(2,3)$ and
$(3,0)$ of $\hat G$ are faces of the resulting map, which is thus locally a
regular square grid around $v$ (up to distance $3$ at least). Then the fact
that faces of color $(1,3)$ and $(0,2)$ also have length $4$ implies that
this map is in fact a four by four toroidal grid. In particular $\hat G$ has only finitely many 
vertices. 

We conclude that there are only finitely many colored graphs with
fixed number of 2-dipoles, hence Proposition~\ref{prop:finitdipole} is proved for $D=3$.

\subsubsection{The case $D\ge 4$.}

The proof of Proposition~\ref{prop:finitdipole} is similar for all $D\ge 4$. 
Consider a melon free $(D+1)$-colored graph $\hat G$ of degree $\delta(\hat G)$ with $2k(\hat G)$ vertices,
having $t_1(\hat G)$ $(D-1)$-dipoles and $t_2(\hat G)$ $(D-2)$-dipoles.  We will show
that the number of such graphs is finite. The bound we establish
below is not tight and can be improved
with minimal effort, but it is sufficient for our purpose.

Let us count faces of degree 2 according to whether they belong to a
$(D-1)$-dipole, a $(D-2)$-dipole or none of these two:
\begin{equation}
  F_1 (\hat G) \le  t_1(\hat G) \binom{D-1}{2} +   t_2(\hat G) \binom{D-2}{2} + \alpha(D) k(\hat G)  \; ,\label{eq:bornesurF1}
\end{equation}
where
\begin{itemize}
 \item $\alpha(4)=0$ as, for $D=4$, all the faces with two vertices
   must belong to a $(D-1)$- or a $(D-2)$-dipole (\emph{i.e.} a 3- or
   a 2-dipole).
 \item $\alpha(5) = 3$ as, for $D=5$, a vertex not belonging to a
   $(D-1)$- or $(D-2)$-dipole can belong to at most three 2-dipoles.
 \item $\alpha(6) = 6$ as, for $D=6$, a vertex not belonging to a
   $(D-1)$- or $(D-2)$-dipole belongs to the largest number of faces
   of degree two when it belongs to two $3$-dipoles \emph{i.e.} to six
   faces of degree two.
 \item $\alpha(D) =\frac{(D-3)(D-4)}{2}+6$, for all $D\ge 7$ as, in
   this case, a vertex not belonging to a $(D-1)$- or $(D-2)$-dipole
   belongs to the largest number of faces of degree two when it
   belongs to a $(D-3)$-dipole and a $4$-dipole.
\end{itemize}
On the one hand, the bound~\eqref{eq:bornesurF1} together with Eq.~\eqref{eq:delta'} gives:
\begin{align*}
   & \sum_{p\ge 2}  [ (D-1)p-D-1 ] F_p (\hat G)  \crcr
   & \qquad \le (D+1) \delta (\hat G) - D(D+1)  \crcr
   & \qquad \qquad + 2 t_1(\hat G) \binom{D-1}{2}  +   2 t_2(\hat G) \binom{D-2}{2}  +  2 \alpha(D) k(\hat G) \; .
\end{align*}
On the other hand, Eq.~\eqref{eq:doublecount1} can be rewritten as $ \frac{D(D+1)}{2} k(\hat G) = F_1(\hat G) + \sum_{p\ge 2} pF_p(\hat G)  $, hence:
\begin{equation}
\label{eq:expr}
 \left[ \frac{D(D+1)}{2} -\alpha(D)  \right] k(\hat G) \le \sum_{p\ge 2} p F_p 
+ t_1(\hat G) \binom{D-1}{2} + t_2(\hat G) \binom{D-2}{2}  \; .
\end{equation}
Eliminating $k(\hat G)$ between these two equations and reordering we get: 
 \begin{align*}
& \sum_{p\ge 2}  \left[ \left( D-1 - \frac{4 \alpha(D)}{D(D+1)-2 \alpha(D) }   \right)p-D-1 \right]F_p(\hat G) \crcr
& \qquad \le  (D+1)\delta -  D(D+1)
+\Bigl( 2 +   \frac{4 \alpha(D)}{D(D+1)-2 \alpha(D) }  \Bigr) t_1(\hat G) \binom{D-1}{2} \\
&\qquad \qquad +  \Bigl(  2 +  \frac{4 \alpha(D)}{D(D+1)-2 \alpha(D) }   \Bigr) t_2(\hat G) \binom{D-2}{2}  \; .
\end{align*}
The coefficient of $F_p$ on the left hand side is 
\begin{itemize}
 \item for $D=4$: $3p-5 $.
 \item for $D=5$: $ \frac{7}{2}p-6$.
 \item for $D=6$: $\frac{21}{5}p-7$.
 \item for $D\ge 7$: $ \frac{3(D-4)(D+1)}{4(D-3)}(p-2) + \frac{(D-6)(D+1)}{2(D-3)}$.
\end{itemize}
In particular this coefficient is strictly positive for $p\ge 2$ so
that we get an upper bound for each $F_p(\hat G)$,
$p\geq2$ depending only on $D$, $\delta(\hat G)$, $t_1(\hat G)$ and $t_2 (\hat G) $, and there is a maximal
value of $p$, depending again only on $D$, $\delta(\hat G)$, $t_1(\hat G)$ and $t_2 (\hat G) $ for which $F_p$
can be non zero.

On the other hand:
\[
  \frac{D(D+1)}{2} -\alpha(D) = \begin{cases}
                                  10\;, \;\; & D=4  \\
                                  12\;, \;\; & D=5  \\
                                  15 \;, \;\; & D=6  \\
                                  4 (D-3)\;, \;\; & D\ge 7  \\
                                \end{cases} \; ,
\]
is always positive, so that from Eq.~\eqref{eq:expr}, we finally get
an upper bound on $k(\hat G)$ depending only on $D$, $\delta(\hat G)$, $t_1(\hat G)$ and $t_2 (\hat G) $.

This completes the proof of Proposition~\ref{prop:finitdipole}.

\newpage

\section{Exact enumeration}\label{sec:enumeration}

\subsection{Melonic graphs and cores}

In view of Theorem~\ref{thm:core} we will need, in order to
enumerate colored graphs,  the generating function of melonic graphs.

\begin{proposition}[See \emph{e.g.} \cite{color}]\label{prop:melonic}
The generating function $T(z)$ of melonic graphs (and open melonic graphs)
with respect to the number of black vertices is
the unique power series solution of the equation:
\[
T(z)=1+zT(z)^{D+1}.
\]
\end{proposition}
\begin{proof}
 Let $M(z)$ be the generating function of prime melonic graphs.
 Then the inductive definition (Definition~\ref{def:melonin}) of melonic graphs and prime
 melonic graphs immediately translate into the equations:
 \[T(z)=1+\sum_{i\geq1}M(z)^i=\frac1{1-M(z)} 
 \; , \qquad 
 M(z)=zT(z)^D \;,
 \]
and we conclude.
\qed
 \end{proof}

\begin{corollary}[\emph{e.g.} \cite{color}]
The generating series $T(z)$ admits the power series expansion
\[
 T(z) = \sum_{k\ge 0}
\frac{1}{(D+1)k+1} \binom{(D+1)k+1}{k} z^k \;.
\]
It has a dominant singularity at $z_0=D^D/(1+D)^{1+D}$ and the
following singular expansion in a slit domain around $z_0$ 
\begin{align} \label{eq:sing}
& T(z) = \frac{1+D}{D}  -\sqrt{ 2 \frac{D+1}{D^3}} \sqrt{ 1 - z/z_0  }+O(1-z/z_0) \; ,
\crcr
&  1-DzT(z)^{D+1} = D \sqrt{ 2 \frac{D+1}{D^3}} \sqrt{ 1 - z/z_0  } +  O(1-z/z_0) \; .
\end{align}
In particular $T(z_0)=(1+D)/D$ and $z_0T(z_0)^{D+1}=1/D$.
\end{corollary}  
  
\begin{proof}
The first expansion follows immediately from
Proposition~\ref{prop:melonic} using Lagrange inversion formula.    The
singular expansion is a direct instance of the standard
theory of singularity analysis of simple trees  generating functions
\cite{flajoletsedgewick}[Chap. VII.4].

\qed
\end{proof}

\begin{proposition}
Let $\hat G$ be a rooted melon-free graph with $2k(\hat G)$ vertices, and
thus $(D+1)k(\hat G)$ edges.  The generating function $H_{\hat G}(z)$ of
colored graphs with core $\hat G$ with respect to the number of
black vertices is
\[
H_{\hat G}(z)= z^{ k( \hat G) }T(z)^{(D+1)k(\hat G) +1}.
\]
\end{proposition}
\begin{proof}
This immediately follows from the bijection of Theorem~\ref{thm:core}. 
The case $k(\hat G) =0$ corresponds to the ring graph which is the core of the melonic graphs.

\qed
\end{proof}
   
\subsection{Chains and schemes}
 
In view of Theorem~\ref{thm:scheme}, in order to enumerate cores
in terms of reduced schemes, we will need several proper chain generating series,
depending on whether the chain is broken or not, on whether its
external edges have identical color or not and on whether its white squares are both on the top or not. Recall that a proper
chain has at least 4 internal vertices.

\begin{description}
\item[\it Arbitrary chains] Let us fix one color $c_1$. 
A non-empty chain with external colors $(c_1,c_1)$ (which is necessarily proper) consists of a non-empty
chain not reusing color $c_1$ followed by a $(D-1)$-dipole with right
external color $c_1$ and a possibly empty chain with external colors
$(c_1,c_1)$:
\begin{align*}
A_{=}(u)& =\frac{Du}{1-(D-1)u}\cdot u\cdot (1+A_{=}) \Rightarrow A_{=}(u)=
\frac{Du^2}{1-(D-1)u}\frac1{1-\frac{Du^2}{1-(D-1)u}} \crcr
& = \frac{Du^2}{(1+u)(1-Du)} =  Du^2  \frac{  1 + (D-1)u - Du^2  }{ (1-u^2) (1-D^2u^2)  }\; .
\end{align*}

Now fix a second color $c_2\neq c_1$. A proper chain
with external colors $(c_1,c_2)$ is
either a dipole with external colors $(c_1,c_2)$ followed by a non empty chain with equal external colors $(c_2,c_2)$,
or a dipole with external colors $(c_1,c'), c' \neq c_2$ followed by either a dipole or a non empty chain with external colors $(c',c_2)$:
\begin{align*}
& A_{\neq} = u A_{=} + (D-1) u^2 + (D-1) u A_{\neq} \crcr
& \Rightarrow A_{\neq} = u^2\frac{ (D-1) + Du}{(1+u)(1-Du)} =  u^2 \frac{ (D-1) + (D^2-D+1) u - D^2u^3 }{(1-u^2)(1-D^2 u^2)} \;.
\end{align*}
The chains with an even number of dipoles correspond to the even powers of $u$, while the ones with an odd number of dipoles to the odd powers of $u$,
hence:
\begin{align*}
& A_{=,\genfrac{}{}{0pt}{}{\circ  \bullet }{   \bullet \circ } } (u) = 
 \frac{ Du^2  (1 - Du^2) }{ (1-u^2) (1-D^2u^2)  } \; ,
\qquad
 A_{=;\genfrac{}{}{0pt}{}{\circ  \circ }{   \bullet \bullet } }(u) = 
\frac{ Du^2  ( D-1 ) u }{ (1-u^2) (1-D^2u^2)  } \; ,\crcr
& A_{\neq;\genfrac{}{}{0pt}{}{\circ  \bullet }{   \bullet \circ } } (u) = 
 \frac{  u^2 (D-1) }{ (1-u^2) (1-D^2u^2)  } \; , \qquad 
 A_{\neq;\genfrac{}{}{0pt}{}{\circ  \circ }{   \bullet \bullet } }(u) = 
\frac{ u^3 \Big[ D^2-D+1  - D^2 u^2 \Big]  }{ (1-u^2) (1-D^2u^2)  } \; .
\end{align*}
\item[\it Unbroken chains] Let us fix two colors $c_1\neq c_2$. There is
  exactly one $(c_1,c_2)$-unbroken chain with $2k$ vertices, for $k\geq1$,
  so that the generating function of proper unbroken chains with respect to the number of
  black vertices, is $ U(u)= u^2 / (1-u)$.
  The half-edges have different colors if the number of dipoles is
  odd, and equal colors if it is even, hence the generating function for the two kinds of proper unbroken chains are:
  \[
  U_{\neq; \genfrac{}{}{0pt}{}{    \circ   \circ }{   \bullet \bullet }  }(u)=\frac{u^3}{1-u^2} \; ,
  \qquad  U_{=; \genfrac{}{}{0pt}{}{\circ  \bullet }{   \bullet \circ } }(u)=\frac{u^2}{1-u^2} \; .
\]
\item[\it Broken chains] Let us fix two colors $c_1$ and $c_2$ (maybe
  equal). A proper broken chain with external colors $(c_1,c_2)$ is an arbitrary
  proper chain which is not unbroken. If $c_1=c_2$, all the $D$ possible second
  colors for the unbroken chain have to be considered: 
  \begin{align*}
  B_{=;\genfrac{}{}{0pt}{}{\circ  \bullet }{   \bullet \circ } }(u)  & = 
  A_{=;\genfrac{}{}{0pt}{}{\circ  \bullet }{   \bullet \circ } } (u) -D
  U_{=;\genfrac{}{}{0pt}{}{\circ  \bullet }{   \bullet \circ } }(u)
   = \frac{ D^2 (D-1) u^4 }{  (1-u^2) (1-D^2u^2)   } \; ,
  \crcr
  B_{=; \genfrac{}{}{0pt}{}{\circ  \circ }{   \bullet \bullet }  }(u)
  & = A_{=; \genfrac{}{}{0pt}{}{\circ  \circ }{   \bullet \bullet }  } (u)  = \frac{ D   ( D-1 ) u^3 }{ (1-u^2) (1-D^2u^2)  } \; ,
  \crcr 
  B_{\neq;\genfrac{}{}{0pt}{}{\circ  \bullet }{   \bullet \circ } }(u)  & =
   A_{\neq;\genfrac{}{}{0pt}{}{\circ  \bullet }{   \bullet \circ } } (u)    = \frac{  (D-1) u^2 }{ (1-u^2) (1-D^2u^2)  } 
  \crcr
  B_{\neq; \genfrac{}{}{0pt}{}{\circ  \circ }{   \bullet \bullet }  }(u)
 & = A_{\neq;\genfrac{}{}{0pt}{}{\circ  \circ }{   \bullet \bullet } }(u) 
    -  U_{\neq; \genfrac{}{}{0pt}{}{\circ  \circ }{   \bullet \bullet }  }(u)
   = \frac{ D(D-1)  u^3 }{ (1-u^2) (1-D^2u^2)  } \; .
  \end{align*}
\end{description}
This is summarized in Fig.~\ref{fig:chains11}.
\begin{figure}[ht]
\begin{center}
\psfrag{c}{$c_1$}
\psfrag{c'}{$c_2$}
\psfrag{Bwwcc1}{{\footnotesize $B_{\neq; \genfrac{}{}{0pt}{}{\circ  \circ }{   \bullet \bullet }  }(u) = 
\frac{ D(D-1)  u^3 }{ (1-u^2) (1-D^2u^2)  } $ } }
\psfrag{Bwwcc}{ {\footnotesize $ B_{=; \genfrac{}{}{0pt}{}{\circ  \circ }{   \bullet \bullet }  }(u)
= \frac{ D   ( D-1 ) u^3 }{ (1-u^2) (1-D^2u^2)  }  $} }
\psfrag{Bwbcc1} { {\footnotesize $  B_{\neq;\genfrac{}{}{0pt}{}{\circ  \bullet }{   \bullet \circ } }(u) =
\frac{  (D-1) u^2 }{ (1-u^2) (1-D^2u^2)  } $ } }
\psfrag{Bwbcc}{ {\footnotesize $ B_{=;\genfrac{}{}{0pt}{}{\circ  \bullet }{   \bullet \circ } }(u) =  
\frac{ D^2 (D-1) u^4 }{  (1-u^2) (1-D^2u^2)   } $} }
\psfrag{Uww} { {\footnotesize $U_{\neq; \genfrac{}{}{0pt}{}{\circ  \circ }{   \bullet \bullet }  }(u)=\frac{u^3}{1-u^2} $ } }
\psfrag{Uwb} { {\footnotesize $ U_{=; \genfrac{}{}{0pt}{}{\circ  \bullet }{   \bullet \circ } }(u)=\frac{u^2}{1-u^2} $ } }
\includegraphics[scale=.4]{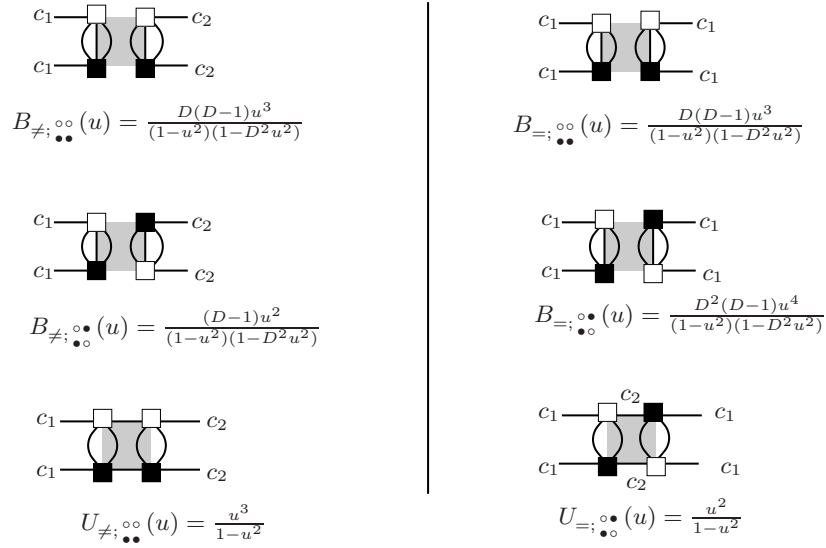}
\end{center}
\caption{The generating series of the chain-vertices.\label{fig:chains11}}
\end{figure}

\begin{proposition}\label{prop:count-core}
Let $\tilde S$ be a reduced scheme with $2k(\tilde S)$ black and white vertices, 
and with $ {\bf U}_{\neq; \genfrac{}{}{0pt}{}{\circ  \circ }{   \bullet \bullet }  }$ chain-vertices 
of type $ U_{\neq; \genfrac{}{}{0pt}{}{\circ  \circ }{   \bullet \bullet }  }$, 
$ {\bf U}_{=; \genfrac{}{}{0pt}{}{\circ  \bullet }{   \bullet \circ } }$ of type 
$ U_{=; \genfrac{}{}{0pt}{}{\circ  \bullet }{   \bullet \circ } }$, 
$ {\bf B}_{\neq; \genfrac{}{}{0pt}{}{\circ  \circ }{   \bullet \bullet }  }$ of type 
$  B_{\neq; \genfrac{}{}{0pt}{}{\circ  \circ }{   \bullet \bullet }  }$, 
$ {\bf B}_{=; \genfrac{}{}{0pt}{}{\circ  \circ }{   \bullet \bullet }  } $ of type 
$ B_{=; \genfrac{}{}{0pt}{}{\circ  \circ }{   \bullet \bullet }  } $, 
$ {\bf B}_{\neq;\genfrac{}{}{0pt}{}{\circ  \bullet }{   \bullet \circ } }  $ of type 
$ B_{\neq;\genfrac{}{}{0pt}{}{\circ  \bullet }{   \bullet \circ } }  $ and 
$ {\bf B}_{=;\genfrac{}{}{0pt}{}{\circ  \bullet }{   \bullet \circ } }  $ of type 
$  B_{=;\genfrac{}{}{0pt}{}{\circ  \bullet }{   \bullet \circ } }  $.
The generating function $G_{\tilde S}(u)$ of rooted melon-free colored graphs with scheme $ \tilde S$ with respect to the number of black vertices is:
\begin{align*}
 G_{\tilde S}(u) & = u^{k(\tilde S)} \; 
 \Bigl[ U_{\neq; \genfrac{}{}{0pt}{}{\circ  \circ }{   \bullet \bullet }  }(u)
 \Bigr]^{  {\bf U}_{\neq; \genfrac{}{}{0pt}{}{\circ  \circ }{   \bullet \bullet }  }  }
 \Bigr[ U_{=; \genfrac{}{}{0pt}{}{\circ  \bullet }{   \bullet \circ } }(u)
 \Bigr]^{  {\bf U}_{=; \genfrac{}{}{0pt}{}{\circ  \bullet }{   \bullet \circ } }   } \times \crcr
& \qquad \qquad \times \Bigl[ B_{\neq; \genfrac{}{}{0pt}{}{\circ  \circ }{   \bullet \bullet }  }(u)
 \Bigr]^{  {\bf B}_{\neq; \genfrac{}{}{0pt}{}{\circ  \circ }{   \bullet \bullet }  } } 
\Bigl[ B_{=; \genfrac{}{}{0pt}{}{\circ  \circ }{   \bullet \bullet }  } (u)
 \Bigr]^{  {\bf B}_{=; \genfrac{}{}{0pt}{}{\circ  \circ }{   \bullet \bullet }  } }
\Bigl[ B_{\neq;\genfrac{}{}{0pt}{}{\circ  \bullet }{   \bullet \circ } } (u)
 \Bigr]^{  {\bf B}_{\neq;\genfrac{}{}{0pt}{}{\circ  \bullet }{   \bullet \circ } } }
\Bigl[ B_{=;\genfrac{}{}{0pt}{}{\circ  \bullet }{   \bullet \circ } } (u)
 \Bigr]^{ {\bf B}_{=;\genfrac{}{}{0pt}{}{\circ  \bullet }{   \bullet \circ } } }
\crcr
 & =    \frac{  P_{ \tilde S}(u)  }{ (1-u^2)^{ {\bf U} + {\bf B} } (1-D^2u^2)^{\bf B} } \; ,
\end{align*}
where ${\bf U} =  {\bf U}_{\neq; \genfrac{}{}{0pt}{}{\circ  \circ }{   \bullet \bullet }  } +
 {\bf U}_{=; \genfrac{}{}{0pt}{}{\circ  \bullet }{   \bullet \circ } }
$, 
${\bf B} = 
 {\bf B}_{\neq; \genfrac{}{}{0pt}{}{\circ  \circ }{   \bullet \bullet }  } +
  {\bf B}_{=; \genfrac{}{}{0pt}{}{\circ  \circ }{   \bullet \bullet }  } + 
   {\bf B}_{\neq;\genfrac{}{}{0pt}{}{\circ  \bullet }{   \bullet \circ } }  +
   {\bf B}_{=;\genfrac{}{}{0pt}{}{\circ  \bullet }{   \bullet \circ } } 
$
and $P_{ \tilde S}(u)$ is the monomial:
\begin{align*}
 P_{\tilde S}(u) & = 
 (D-1)^{ {\bf B}   }  \;\;  D^{ {\bf B}_{\neq; \genfrac{}{}{0pt}{}{\circ  \circ }{   \bullet \bullet }  } +
  {\bf B}_{=; \genfrac{}{}{0pt}{}{\circ  \circ }{   \bullet \bullet }  } + 
  2 {\bf B}_{=;\genfrac{}{}{0pt}{}{\circ  \bullet }{   \bullet \circ } }  } \crcr
& \qquad  \times u^{k(\tilde S) +2 {\bf U}_{=; \genfrac{}{}{0pt}{}{\circ  \bullet }{   \bullet \circ } } +
  3 {\bf U}_{\neq; \genfrac{}{}{0pt}{}{\circ  \circ }{   \bullet \bullet }  }
+ 3 {\bf B}_{\neq; \genfrac{}{}{0pt}{}{\circ  \circ }{   \bullet \bullet }  } +
  3{\bf B}_{=; \genfrac{}{}{0pt}{}{\circ  \circ }{   \bullet \bullet }  } + 
  2 {\bf B}_{\neq;\genfrac{}{}{0pt}{}{\circ  \bullet }{   \bullet \circ } }  +
  4 {\bf B}_{=;\genfrac{}{}{0pt}{}{\circ  \bullet }{   \bullet \circ } } 
 }   \; .
\end{align*}
\end{proposition}

\begin{proof}
This follows immediately from the bijection in Theorem~\ref{thm:scheme}.

\qed
\end{proof}

\subsection{The enumeration of rooted colored graph of fixed degree}

Putting together Theorem~\ref{thm:core}, Theorem~\ref{thm:scheme} and Theorem~\ref{thm:finiteness} 
we obtain the enumeration of the edge colored graphs of fixed degree.
\begin{theorem}\label{thm:count-all}
Let $\delta\geq 0$.  The generating function of rooted colored graphs  with root edge of color $0$
and degree $\delta$ with respect to the number of black vertices is
\[
H^0_\delta(z)=T(z)\sum_{ \tilde S \in \tilde {\cal S}^0_\delta} G_{ \tilde S} \left( zT(z)^{D+1} \right)
\]
where the sum runs over the finite set $ \tilde {\cal S}^0_\delta$ of reduced schemes with degree $\delta$ and root edge of color $0$.
\end{theorem}

Together with Proposition \ref{prop:count-core}, this theorem implies Theorem \ref{thm:main1} in the Introduction.
The first values can be computed explicitly. 

\subsubsection{Degree $\delta=0$} The reduced schemes of degree zero have no chain-vertex. 
First they can not have any non-separating chain-vertex. Assume now that they have 
separating chain-vertices. Deleting the chain vertices one obtains several connected components 
which are (see Section~\ref{sec:prf1} ) colored graphs with only black marks. All these connected components have degree $0$, and some of them have only one mark, which is impossible
according to Lemma~\ref{lem:zerodegcomp}. It follows that the reduced schemes of degree zero are melon-free colored graphs of degree zero, hence the unique such reduced scheme is 
the ring graph and, as expected: 
\[
 H^0_0(z) = T(z) \; .
\]

\subsubsection{Degree $\delta = D-2$}

We are interested in identifying the smallest integer $\delta_{\rm min}>0$ such that there exist reduced schemes $\tilde S_{\rm min}$ 
with $\delta(\tilde S_{\rm min}) = \delta_{\rm min}$ and furthermore classify all the reduced schemes $\tilde S_{\rm  min}$.

\begin{lemma}\label{lem:maxofmin}
 The minimal non-zero degree $\delta_{\rm min}$ is at most $D-2$.
\end{lemma}
\begin{proof} For any $D$, let us consider the two reduced schemes presented in Fig.~\ref{fig:lolly}
where the vertical arm can be empty, represent a dipole or a chain vertex (broken or unbroken), and the unbroken chain vertex can be replaced by a 
unique dipole for the case on the left hand side.
  \begin{figure}[ht]
\begin{center}
\psfrag{c1}{$c_1$}
\psfrag{c2}{$c_2$}
\includegraphics[scale=.4]{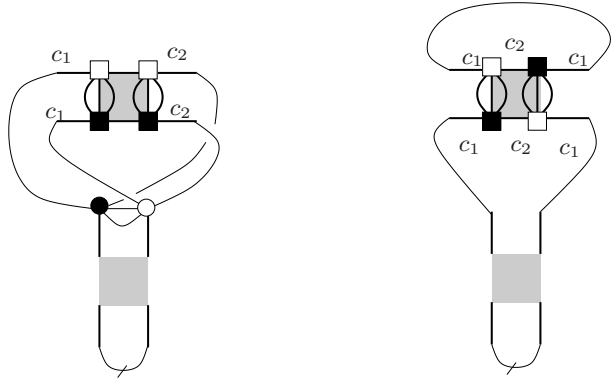}
\end{center}
\caption{The ``lollipop'' reduced schemes.\label{fig:lolly}}
\end{figure}

These schemes have a non-separating unbroken chain-vertex in Case III.b (single face of colors $(c_1,c_2)$ after deletion), hence
the degree goes down by $D-2$ upon deletion and the resulting scheme (which is no longer reduced) has degree $0$.

\qed
\end{proof}

We will now show that the minimal non trivial degree is $D-2$ and that the ``lollipop'' schemes in Fig.~\ref{fig:lolly} are the only reduced schemes of degree $D-2$.
Let us denote $\hat S_{\rm min}$ the melon free colored graph with degree $\delta_{\rm min}$ obtained from $\tilde S_{\rm min}$ by replacing all the chain-vertices by their minimal realization
as chains of $(D-1)$-dipoles.

\begin{lemma}\label{lem:whatdipole1} $\hat S_{\rm min}$ can not contain:
\begin{itemize}
 \item $\forall D\ge 3$, two separating $(D-1)$-dipoles which do not belong to the same maximal chain.
 \item $\forall D\ge 4$  a (partially) separating $(D-q)$-dipole with $2 \le q \le D-2$ which separates it in more than two connected components.
 \item $\forall D\ge 5$ a non-separating $(D-2)$-dipole.
 \item $\forall D\ge 5$ a $2$-dipole which separate $\hat S_{\rm  min}$ into exactly two connected components.
 \item $\forall D\ge 6$ a $(D-q)$-dipole with $3 \le q \le D-3$.
\end{itemize}
\end{lemma}
\begin{proof} Assume that $\hat S_{\rm min}$ contains a $(D-q)$ dipole. By removing the dipole,  $\hat S_{\rm min}$ separates into $C$ connected components
and from Eq.~\eqref{eq:variationbydeletion} the variation of the degree through the removal is at least: 
\[ D(q+1-C) - t_1(t_1-1) -\dots t_C(t_C-1) \; ,\]
where $t_i\ge 1$ denotes the number of new edges in the component $i$ (hence $t_1+\dots +t_C =q+1$). This bound is saturated only if, in each connect component,
all the faces containing new edges are of Type b, \ie any pair of new  edges in the same connected component belongs to the same face after deletion.

{\bf First item.} By removing two separating $(D-1)$-dipoles which do not belong to the same maximal chain, $\hat S_{\rm min}$ separates into 
three connected components. But this impossible as at most one of the components can have zero degree (the one containing the root), and $\delta_{\rm  min}$
can not be distributed among the two remaining components.

{\bf Second item.} By removing such a $(D-q)$-dipole, $\hat S_{\rm min}$  splits into at least three connected components which, by the same argument as before, is impossible.

{\bf Third item.} In the case $q=2$, $C=1$ (hence $t_1 = 3$),  the variation of the degree is at least $2D-6 > D-2, \; \forall D\ge 5$.

{\bf Fourth item.} In the case $q=D-2$, $C=2$ the variation of the degree is at least:
\[
 D (D-3) - t_1(t_1-1) - t_2 (t_2-1) \; ,
\]
which, using $t_1 + t_2 = q+1=D-1$, amounts to $ -2t_1^2 + 2 (D-1) t_1 -2$.
Taking into account the sign of $t_1^2$, we have:
\begin{align*}
& \min_{t_1\in \{1,\dots,D-2\} } \left\{   -2t_1^2 + 2(D-1) t_1 -2 \right\} = \min_{t_1 = 1, t_1 = D-2 } \left\{   -2t_1^2 + 2(D-1) t_1 -2  \right\} =\crcr
& \qquad \qquad  = 2D-6 >D-2 \;, \qquad \forall D\ge 5 \;.
\end{align*}
 
{\bf Fifth item.} Let us now consider the variation of the degree with the deletion of a $(D-q)$-dipole in the only two possible cases, $C=1,2$:
\begin{itemize} 
\item $C=1$, (hence $t_1 = q+1$). The variation of the degree with the deletion is $D q - q(q+1)$,
and, due to the sign of $q_2$, we have:
\begin{align*}
&  \min_{q \in \{ 3,\dots D-3\} } \left\{ -q_2 + (D-1)q \right\} = \min_{q=3,q= D-3} \left\{ -q_2 + (D-1)q \right\} = \crcr
& \qquad \qquad = \min \{3D-12, 2D-6 \} > D-2  \;, \qquad \forall D\ge 6 \;.
\end{align*}
\item $C=2$. The variation of the degree is $  D(q-1) - t_1 (t_1-1) - t_2(t_2-1)$, which, using $t_1 + t_2 = q+1$, amounts to
$-2t_1^2 + 2(q+1) t_1 + D(q-1) -q(q+1)$.
Taking into account the sign of $t_1^2$, we have:
\begin{align*}
& \min_{t_1\in \{1,\dots,q\} } \left\{   -2t_1^2 + 2(q+1) t_1 + D(q-1) -q(q+1)  \right\} = \crcr
& \qquad = \min_{t_1 = 1, t_1 = q } \left\{   -2t_1^2 + 2(q+1) t_1 + D(q-1) -q(q+1)  \right\} = -q^2 + D(q-1) + q \;,
\end{align*}
and, taking the minimum over $q$, we have:
\begin{align*}
&   \min_{q \in \{ 3,\dots D-3\} } \left\{ -q^2 + D(q-1) + q  \right\} = \min_{q=3,q= D-3} \left\{ -q^2 + D(q-1) + q   \right\} =
   \crcr
& \qquad \qquad = \min \{ 2D-6 , 3D-12\} > D-2  \;, \qquad \forall D\ge 6 \; .
\end{align*}
\end{itemize}
which concludes.

 \qed
\end{proof}

\begin{lemma}\label{lem:whatdipole2}
For $D\ge 4$, if $\hat S_{\rm min}$ contains a non-separating $2$-dipole, then it contains a non-separating $(D-1)$-dipole.
\end{lemma}
\begin{proof} 

We will show that if  $\hat S_{\rm min}$  contains a non-separating $2$-dipole, then it it is one of the two graphs depicted in Fig.~\ref{fig:problemsolution},
(where the vertical arm can be empty or represent a chain of $(D-1)$-dipoles) hence contains a non-separating $(D-1)$-dipole.

\begin{figure}[ht]
\begin{center}
\psfrag{c1}{$c$}
\psfrag{c2}{$c'$}
\includegraphics[scale=.3]{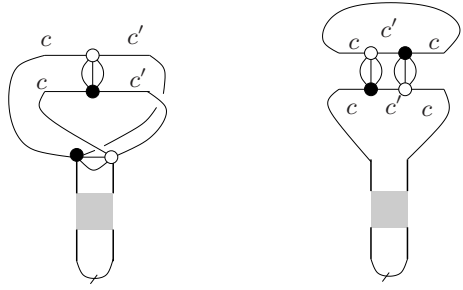}
\end{center}
\caption{Graphs with a non-separating $2$-dipole.\label{fig:problemsolution}}
\end{figure}

In the notations of the previous lemma, this corresponds to $q=D-2$ for and $C=1$. 
By deleting the $2$-dipole the degree goes down by at least $D-2$, and it goes down by exactly $D-2$ only if
for any two external colors $c_1,c_2$ of the $2$-dipole, all the faces $(c_1,c_2)$ incident to the $2$-dipole are of Type b 
(single face after the deletion).

After the deletion of the $2$-dipole one obtains a graph $\hat S'$ of degree zero. 
Inspired by Section \ref{sec:prf1}, let us mark the $D-1$ new edges of $\hat S'$ generated by the deletion with blue marks: any couple of marked edges in $\hat S'$
belong to a face. Observe that if the root of $\hat S_{\rm min}$ is incident to the dipole, then the root of $\hat S'$ is marked.  
While $\hat S'$ is melonic and can have melonic subgraphs, because $\hat S_{\rm min}$ is melon free, any melonic subgraph of $\hat S'$ 
must contain at least a marked edge.

If $\hat S'$ has two vertices, it is immediate to see that we are in one of the two cases depicted in Fig.~\ref{fig:problemsolution} with empty vertical arm, 
depending on whether the root of $\hat S'$ is marked or not.

If $\hat S'$ has four vertices or more, then it has at least two pairs of vertices, $(u_{\circ}, u_{\bullet})$ and $(v_{\circ}, v_{\bullet})$ connected by $D$ parallel edges, hence 
at least one pair, say $(v_{\circ}, v_{\bullet})$, connected by $D$ parallel non-root edges. 
Then at least one of the parallel edges connecting $v_{\circ}$ and $v_{\bullet}$ must be marked.
This edge belongs to $D-1$ faces of degree $2$ made only of parallel edges connecting $v_{\circ}$ and $v_{\bullet}$, and to only one face of degree larger than two.
Thus at least $D-3$ other parallel edges must be marked.

This implies that $\hat S'$ can not have a third pair of vertices $ (w_{\circ}, w_{\bullet}) $ connected by $D$ parallel edges, as this
would require another $D-2$ marked edges and, as $D\ge 4$, $2D-4>D-1$. It follows that $\hat S'$ is a chain of $(D-1)$-dipoles such that:
\begin{itemize}
 \item the left half-edges (incident to $u_{\circ}$ and $u_{\bullet}$) are joined together into the root edge,
 \item the chain has at least two $(D-1)$-dipoles $u_{\circ} - u_{\bullet} $ and  $v_{\circ} - v_{\bullet}$,
 \item the right half-edges are joined  into an edge,
 \item at least $D-2$ of the parallel edges connecting $v_{\circ}$ and $v_{\bullet}$ are marked.
\end{itemize}

Finally, the last marked edge either connects also $v_{\circ}$ and $v_{\bullet}$, or it is incident to one of them (say $v_{\bullet}$), 
as only these edges in the chain share faces with $D-2\ge 2$ of the parallel edges connecting  $v_{\circ}$ and $v_{\bullet}$.
This is depicted in Fig.~\ref{fig:almostthere}.
\begin{figure}[ht]
\begin{center}
\psfrag{vc}{$v_{\circ}$}
\psfrag{vb}{$v_{\bullet}$}
\includegraphics[scale=.4]{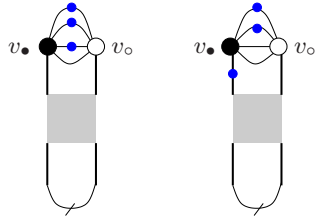}
\end{center}
\caption{Blue marks after deletion of a non-separating $2$-dipole.\label{fig:almostthere}}
\end{figure}
Reinstating the $2$-dipole leads to the two cases in Fig.~\ref{fig:problemsolution}.

 \qed
\end{proof}

\begin{lemma}\label{lem:whatdipole3}
For $D\ge 4$, if $\hat S_{\rm min}$ contains a $(D-2)$-dipole which separates it into two connected components, then it contains a non-separating $(D-1)$-dipole.
\end{lemma}
\begin{proof} 
Let us remove the $(D-2)$-dipole and mark the new edges with blue marks. The graph $\hat S_{\rm min}$ separates into two connected components
of degree zero such that one component has two marks and the other one has one.

The connected component not containing the root edge $r(\hat S_{\rm min})$ has at least two vertices and, as it does not have any melonic subgraph made only of unmarked edges,
it must contain two marked edges. It follows that this component is a chain having at least a $(D-1)$-dipole with left 
half-edges and right half-edges joined into (marked) edges. Any $(D-1)$-dipole in the chain is a non-separating
dipole of $\hat S_{\rm min}$.
\end{proof}

\begin{proposition}\label{prop:whatdipole3}
A graph $\hat S_{\rm min}$ must contain a Case II.B non-separating $(D-1)$-dipole (single resulting face after the removal).
\end{proposition}
\begin{proof}

First observe that it suffices to show that $\hat S_{\rm min}$ contains a non-separating $(D-1)$-dipole: as 
$\delta_{\rm min}\le D-2$, this dipole can only be a Case II.B non-separating dipole.
If $\hat S_{\rm min}$ contains a separating $(D-1)$-dipole, removing the maximal separating chain to which this dipole belongs,
$\hat S_{\rm min}$ splits into two graphs, one of which has degree $\delta_{\rm min}$ and, from Lemma \ref{lem:whatdipole1},
does not contain any separating $(D-1)$-dipole. It thus suffices then to show that any $\hat S_{\rm min}$ contains 
at least a $(D-1)$-dipole. According to Eq.~\eqref{eq:delta'}:
\[ (D+1)\delta_{\rm  min} + 2F_1(\hat S_{\rm min})  =  D(D+1)+ \sum_{p\geq
    2} \left[(D-1)p-D-1 \right]F_p(\hat S_{\rm min})  \ge D(D+1)\; , \]
and, as $\delta_{\rm min}\le D-2$, $F_1( \hat S_{\rm min}) \ge D+1$. Thus $\hat S_{\rm min}$ has at least a face of degree $2$ which does not contain the root edge,
hence belongs to some $(D-q)$-dipole for $1\le q\le D-2$. We denote this dipole $d$. We have:
\begin{itemize}
 \item for $D=3$, $d$ is a $2$-dipole, hence a $(D-1)$-dipole.
 \item for $D=4$, $d$ is either:
  \begin{itemize}
    \item a $2$-dipole, hence a $(D-2)$-dipole. Then $d$ can not be completely separating (Lemma~\ref{lem:whatdipole1}), 
         and if it is either non-separating (Lemma~\ref{lem:whatdipole2}) or partially separating (Lemma~\ref{lem:whatdipole3}),
         then $\hat S_{\rm min}$ contains a $(D-1)$-dipole.
    \item a $3$-dipole, hence a $(D-1)$-dipole.
  \end{itemize}
 \item for $D=5$, $d$ is either:
      \begin{itemize}
       \item a $2$-dipole. According to Lemma~\ref{lem:whatdipole1}, $d$ must be non-separating hence (Lemma~\ref{lem:whatdipole2}) $\hat S_{\rm min}$ contains a $(D-1)$-dipole.
       \item a $3$-dipole, hence a $(D-2)$-dipole. Then (Lemma~\ref{lem:whatdipole1}) $d$ must be partially separating and (Lemma~\ref{lem:whatdipole3}) $\hat S_{\rm min}$ contains a $(D-1)$-dipole.       
       \item a $4$-dipole, hence a $(D-1)$-dipole.
      \end{itemize}
 \item for $D\ge 6$, $d$ is either:
      \begin{itemize}
       \item a $2$-dipole. According to Lemma~\ref{lem:whatdipole1} $d$ must be non-separating hence (Lemma~\ref{lem:whatdipole2}) $\hat S_{\rm min}$ contains a $(D-1)$-dipole.
       \item a $(D-2)$-dipole. Then, from Lemma~\ref{lem:whatdipole1}, $d$ must be partially separating hence (Lemma~\ref{lem:whatdipole3}) $\hat S_{\rm min}$ contains a $(D-1)$-dipole.       
       \item a $(D-1)$-dipole.
      \end{itemize}
\end{itemize}
This concludes the proof. 

\qed
\end{proof}

We are finally in the position to classify the minimal reduced schemes and to determine the minimal non zero degree.

\begin{theorem}\label{thm:lolly}  
The minimal non zero degree is $\delta_{\rm min} = D-2$ and 
the only reduced schemes of minimal non zero degree  ($\tilde S_{\rm  min}$ with $\delta( \tilde S_{\rm  min}) = \delta_{\rm min}$) are  the lollipop schemes represented in Fig.~\ref{fig:lolly}.
\end{theorem}
\begin{proof} 

From Proposition~\ref{prop:whatdipole3}, the minimal realization $\hat S_{\rm min}$ of $\tilde S_{\rm min}$ as a colored graph
contains a non-separating  Case II.B $(D-1)$-dipole which we denote $d$. Deleting the \emph{maximal} (necessarily unbroken) chain $U_d$ containing $d$ in $\hat S_{\rm min}$,
the degree decreases by exactly $D-2$, hence $\delta_{\rm min}\ge D-2$ and in conjunction 
with Lemma \ref{lem:maxofmin} we conclude that $\delta_{\rm min} = D-2$.

The deletion leads to a colored graph $\hat S'_{\rm min}$ of degree 0 having two blue marks, 
such that any melonic subgraph of $\hat S'_{\rm min}$ contains at least a blue mark.

If $\hat S'_{\rm min}$ has only one edge carrying both marks, then it is either a ring or a chain with the left halfedges matched into the root edge and the right halfedges matched into the marked edge.
It follows that $\tilde S_{\rm  min}$ is of the form depicted on the right in Fig.~\ref{fig:lolly}.

Suppose now that $\hat S'_{\rm min}$ has two edges carrying the marks. Then $\hat S'_{\rm  min}$ must have at least a pair of vertices connected by $D$ parallel non-root edges. 
As any melonic subgraph of $\hat S'_{\rm  min}$ must have a blue mark, one of the 
$D$ parallel non-root edges is marked. If no other parallel non root-edge is marked, it follows that the $D-1$ parallel edges 
are a $(D-1)$-dipole in $\hat S_{\rm min}$ which extends the chain $U_d$. But this is not possible as $U_d$ is maximal. 
Hence both marked edges connect the same two vertices. It follows that $\hat S'_{\rm min}$ is a chain with left half-edges matched in the root edge,
right half-edges matched in an edge such that two internal edges of the rightmost dipole are marked
and $\tilde S_{\rm  min}$ is of the form depicted on the left in Fig.~\ref{fig:lolly}.

\qed
\end{proof}

Let us analyze first the vertical arm in Figure~\ref{fig:lolly}. Including the degenerate configurations, it can be empty, 
consist of one dipole, or consist of a (broken or unbroken) chain-vertex. Fixing the incoming color (bottom color of the arm in 
Figure~\ref{fig:lolly}) to be $c_1$, allowing all the possible outgoing colors $c_2$, and taking into account that, if the chain-vertex is unbroken
and has external colors $(c_1,c_1)$, then it can have any secondary color $c_2\neq c_1$, the generating function of the vertical arm 
in Figure~\ref{fig:lolly} is:
\begin{align*}
& 1   + B_{=;\genfrac{}{}{0pt}{}{\circ  \bullet }{   \bullet \circ } }(u)  
 +   D U_{=;\genfrac{}{}{0pt}{}{\circ  \bullet }{   \bullet \circ } }(u) 
 +  B_{=; \genfrac{}{}{0pt}{}{\circ  \circ }{   \bullet \bullet }  }(u)
+  D [ u+ B_{\neq; \genfrac{}{}{0pt}{}{\circ  \circ }{   \bullet \bullet }  }(u) 
  +   U_{\neq; \genfrac{}{}{0pt}{}{\circ  \circ }{   \bullet \bullet }  }(u) 
  + B_{\neq;\genfrac{}{}{0pt}{}{\circ  \bullet }{   \bullet \circ } }(u) 
  ] \crcr
& \qquad = 1  + \frac{Du^2(1-Du^2)}{(1-u^2)(1-D^2u^2)} + 
\frac{ D   ( D-1 ) u^3 }{ (1-u^2) (1-D^2u^2)  }+ \crcr
& \qquad \qquad  + D \left[  \frac{ u(1-Du^2)}{  (1-u^2)(1-D^2u^2)  }+  \frac{  (D-1)u^2}{  (1-u^2)(1-D^2u^2)  }\right] \crcr
& \qquad =  1 + \frac{Du^2}{ (1+u)(1-Du) } + \frac{Du}{ (1+u)(1-Du) } =\frac{1}{1-Du} \; .
\end{align*}

The generating function of the lollipop schemes are therefore:
\begin{align*}
&  T(z) \left[   \frac{1}{1-Du} \cdot u \cdot \binom{D}{2}  \cdot \frac{u}{1-u^2}\right]_{u= zT(z)^{D+1}} \; , \crcr
&  T(z) \left[  \frac{1}{1-Du} \cdot D \cdot \frac{u^2}{1-u^2}\right]_{u= zT(z)^{D+1}} \; ,
\end{align*}
where we counted the fact that the non separating chain-vertex can be reduced to a unique dipole for the leftmost scheme.
Putting the two together we obtain:
\begin{align*}
 H^0_{D-2}(z) & = T(z) \left[ \frac{1}{1-Du}    \left[ \binom{D}{2} +D \right]   \frac{u^2}{1-u^2}\right]_{u= zT(z)^{D+1}} \crcr
 & =  T(z) \frac{D(D+1)}{2}  \frac{ z^2 T(z)^{2D+2} }{ 1 - z^2 T(z)^{2D+2} } \frac{1}{1-DzT(z)^{D+1}} \; ,
\end{align*}
reproducing the result of \cite{ryanoriti}.  An alternative proof
would have been to list all the reduced schemes of
$\bigcup_{\delta\geq1} \tilde {\cal S}^0_\delta$ up to a sufficiently
large size and compute their degree and contribution: this is
admittedly quite tedious by hand, but could in principle easily lead
to an automated computation of the $H^0_\delta$  for small $\delta$.

\newpage

\section{Dominant schemes}\label{sec:asymptotics}

Combining Proposition~\ref{prop:count-core} and Theorem~\ref{thm:count-all}
with the singular expansion~\eqref{eq:sing}, we immediately see that 
$H^0_\delta(z)$ has radius of convergence $z_0$ and admits a singular
expansion near $z=z_0$ of the form: 
\[
H^0_\delta(z)= \frac{D+1}{D}\sum_{ \tilde S\in \tilde {\cal S}^0_\delta}\frac{ P_{\tilde S}( 1/ D)}{(1-1/D^2)^{{\bf U}+{\bf B}}[ 2 (D+1)/D^3]^{{\bf B}/2} 2^{\bf B} }(1-z/z_0)^{-{\bf B}/2} \left[ 1+O \left(\sqrt{1-z/z_0} \right) \right] \;,
\]
which is dominated by the reduced schemes that maximize ${\bf B}$. In other terms, with
probability tending to 1 when $k$ goes to infinity, a uniform random
colored graph with degree $\delta$ will have a reduced scheme that maximizes
${\bf B}$.
In order to identify these schemes we improve in this section the
analysis of the number of broken chain-vertices in a reduced scheme.

\subsection{A bound on the number of broken chain-vertices}

Let us use a simplification of the algorithm presented in Section~\ref{sec:iterate} and only remove broken chain-vertices in a reduced scheme $\tilde G$.
As the first step involves only non separating deletions, $\tilde G'$ is connected and every connected component of $\tilde G''$ has at least a black mark (or no separating deletion 
is performed). The algorithm goes through with several modification: 
\begin{itemize}
 \item the number of blue marks in $\tilde  G'$ and $\tilde G''$ is $m_{\rm blue} (\tilde G') = m_{\rm blue} (\tilde G'') =  2 q $, where $q$ is the number of non-separating deletions. 
 \item  the degree goes down by exactly $D$ with each non separating deletion $\delta( \tilde G) = \delta(\tilde G')+ D q$.
 \item the abstract graph $\mathfrak{F}$ associated to $\tilde G''$ is a tree $\mathfrak{T}$ and $m_{\rm black} ( \tilde G'') = 2 E(\mathfrak{T})$.    
 \item in view of Theorem~\ref{thm:lolly}, the minimal degree of a positive degree component
is $\delta_{\rm min}=(D-2)$. Denoting $c_+$ the number of such components, we have:
 \[
  (D-2) c_+ + D q  \le \delta( \tilde G)
 \]
\end{itemize}

Crucially, we have the following result.

\begin{lemma} The non root connected components of degree zero in $\tilde G''$ must have at least three marks (either blue or black). 
The non root connected components of degree zero with exactly three marks are either ring components 
or consist of two vertices connected by $D$ parallel edges, three of which are marked.
 \end{lemma}
 
 \begin{proof}
  We denote $\hat G$, $\hat G'$ and $\hat G''$ the minimal realizations of the schemes $\tilde G$, $\tilde G'$ and $\tilde G''$ as colored graphs. As
  $\tilde G$ is reduced, $\hat G$ is melon free.
  
Let us consider a non root zero degree component of $\hat G''$, say $\hat G_1$. Then $\hat G_1$ is a melonic graph such that any of its melonic subgraphs contains at least a marked edge.
Observe that $\hat G_1$ has at least a mark coming from deleting some separating maximal chain in $\hat G'$ (\ie some maximal chain in $\hat G$).
But then it must have at least another mark, as $\hat G$ has no melonic subgraph.

If $\hat G_1$ is a ring graph and has exactly two marks, then it corresponds to two chain-vertices in the scheme $\tilde G$ joined together into a longer chain, which is impossible
as $\tilde G$ is reduced.

If $\hat G_1$ is non trivial then, for any $D$-uple of parallel edges connecting the same two vertices in $\hat G_1$:
\begin{itemize}
 \item at least one of the parallel edges is marked (as any melonic subgraph of $\hat G_1$ has a marked edge),
 \item if one of the parallel edges has only one mark then at least another parallel edge has a mark, 
 (otherwise in $\tilde G$ the left (or right) half-edges of a chain-vertex would be incident to the left (or right) half-edges of a 
 $(D-1)$-dipole or chain-vertex).
\end{itemize}

If $\hat G_1$ has four vertices or more, it has two disjoint $D$-uples of parallel edges and, as there are at least two marks
for any $D$-uple of parallel edges, $\hat G_1$ has at least four marks.

If $\hat G_1$ has only two vertices joined by $D+1$ parallel edges (say $e^0,\dots e^D$), at least one of them (say $e^0$) is marked. 
Considering the $D$-uple $e^1, \dots e^D$, we conclude that one of these edges (say $e_1$) is also marked. 
If either $e^0$ and $e^1$ have only one mark, then considering the $D$-uple $e^0, e^2,\dots e^D$ or $e^1,e^2,\dots e^D$ we conclude that 
one of the edges $e^2,\dots e^D$ is marked. Thus either $\hat G_1$ has three marks on three distinct edges, or it has at least four marks.

\qed
\end{proof}
  
Let us denote $c_0$ the number of non root zero degree components (each of which has at least three marks) of $\tilde G''$. The positive degree components and the root component have at least a mark.
As $\mathfrak{T}$ is a tree we have:
\[
 1 + c_+ + c_0 =  E(\mathfrak{T}) +1 \Rightarrow  m_{\rm black} ( \tilde G'') = 2 c_+ + 2 c_0 \;,
\]
and counting the minimal number of marks in a connected component we have:
\[
 m_{\rm blue} (\tilde G'') + m_{\rm black} (\tilde G'') \ge 1 + c_+ + 3c_0 \; ,
\]
which implies:
\[
 2q  +  2 c_+ + 2 c_0 \ge 1 +  c_+ + 3c_0 \Rightarrow 2q  +    c_+ -1 \ge  c_0 \; .
\]

Moreover, this inequality is saturated only if all the zero degree non root components have exactly three marks, and all the positive degree components
and the root component have exactly one mark. The total number of broken chain-vertices is half the number of marks, hence:
\[
 {\bf B} = \frac{ m_{\rm blue} (\tilde G'')  + m_{\rm black}(\tilde G'')  }{2} = q  + c_+ + c_0 \le  3q  + 2   c_+ -1 \;.
\]
\begin{proposition}
The number ${\bf B}$ of broken chain-vertices in a reduced scheme $\tilde G$ of degree $\delta$ is at most:
\[
   {\bf B}_{\max} = 2 c_+  +3q  -1 \; ,
\]
and it saturates this bound only if in $\tilde G''$ all the positive degree components and the root component  have
exactly one mark and all the non root zero degree components have exactly three marks.

Furthermore, the parameters $c_+$ and $q $ satisfy:
\[
(D-2) c_+  + D q \leq \delta \; .
\]
and the constraint is saturated only if all the components of positive
degree of $\tilde G''$ have degree exactly $D-2$.
\end{proposition}

\subsection{Realizability and dominant schemes}

Given integers $c_+$ and $q$ that satisfy the constraint
$(D-2)c_++Dq\leq\delta$, it is always possible to construct a reduced scheme
with these parameters that has $2c_++3q-1$ broken chain-vertices: form
$c_+$ loops each using one unbroken chain-vertex, and put these loops
and a root ring component at the $c_++1$ leaves of a binary tree whose $2c_+-1$
edges are separating broken chain-vertices. Finally add $q$ non-separating broken chain-vertices
and attach their two extremities to existing edges: each attachment
creates another broken chain-vertex (as it splits an existing one  into two), so
that $3q$ broken chain-vertices are added in total. The total number of broken chain-vertices is thus 
of $2c_++3q-1$.

\begin{proposition}
For any $\delta\geq1$, the dominant reduced schemes $\tilde G$ of degree $\delta$ are reduced schemes
with ${\bf B}_{\max}$ broken chains where  ${\bf B}_{\max}$ is the maximum of the 
integer linear program:
\[
2 c_+ +3 q-1, \textrm{ subject to the constraint } (D-2) c_+ +Dq \le \delta \;,
\]
and $\tilde G$ is such that:
\begin{itemize}
\item All the non root zero degree components in $\tilde G''$ have exactly 3 marks: they are either ring components with three marks, or 
consist in two vertices connected by $D+1$ parallel edges, three of which are marked.
\item The root component has one mark and degree 0: it is a root ring graph.
\item All the strictly positive degree components in $\tilde G''$ have one mark and degree
  $D-2$. They either represent an unbroken chain with equal external colors 
  and half-edges paired into edges $\Braket{\ell_{\bullet}, r_{\circ}}$ and $\Braket{r_{\bullet}, \ell_{\circ}}$, one of which is marked
  or they represent an unbroken chain with different external colors incident at a $(D-2)$-dipole whose third pair of half-edges is joint into a marked edge.
\item All the other elements of the schemes are ${\bf B}_{\max}$ broken chains.
\end{itemize}
\end{proposition}

Together with the asymptotic expansion~\eqref{eq:sing}, this Proposition implies Theorem \ref{thm:main2} in the Introduction.
The dominant reduced schemes can be seen as
abstract graphs whose edges represent the broken chains and whose vertices can
either be trivalent (representing the non-root zero degree components) or
univalent (representing root component and the non zero degree components).

In order to completely characterize these schemes we need to determine
${\bf B}_{\max}$. Optimizing the above linear program yields the parameters
of the candidate reduced schemes with a maximal number of broken
chain-vertices: the ``pure'' solutions are
\begin{eqnarray}
c^a_+ =\delta/(D-2) \;,\;\; q_a =0 & \Rightarrow & {\bf B}^a = 2\delta/(D-2)-1\\
c^b_+ = 0 \; , \;\; q_b=\delta/D & \Rightarrow & {\bf B}^b = 3\delta/D-1\\
\end{eqnarray}
Pure solutions are not realizable for values of $\delta$ that are
not divisible by $D$ or $D-2$, so that for $D\geq5$ mixed solutions
should be considered also. 

Let us describe the schemes in the pure cases:
\begin{itemize}
\item Case $a$, with $\delta=n\cdot(D-2)$, $n\ge 1$: since $q_a=0$, all the
  $2n-1$ broken chains are separating: the scheme is a binary tree
  with $n+1$ leaves, one carrying the root ring and the others
  carrying unbroken loops, with $n-1$ internal nodes each carrying a
  ring or a $(D-2)$-dipole, and with the $2n-1$ edges carrying the broken
  chains. The generating function of graphs associated to such schemes is obtained by counting: 
 \begin{itemize}
   \item the root ring contributes 1.
   \item counting the choices of colors, every non root leaf contributes:
   \[
   D U_{=;\genfrac{}{}{0pt}{}{\circ  \bullet }{   \bullet \circ } }(u)  + \binom{D}{2} u \left( u +   U_{\neq; \genfrac{}{}{0pt}{}{\circ  \circ }{   \bullet \bullet }  }(u)  \right) = \frac{D(D+1)}{2} \frac{u^2}{1-u^2} \; .
   \]
    \item counting the choices of the colors, every trivalent internal node contributes:
   \[
    1 + \binom{D}{2} u \;,
   \]
   where the $1$ corresponds to the case  of a ring component with three marks, and second term corresponds
   to the case of a $D-2$-dipole with all the half-edges matched into marked edges.
   \item  counting the choices of outgoing colors, the generating function for the separating chains with fixed incoming color is:
 \begin{align*}
&  B_{=;\genfrac{}{}{0pt}{}{\circ  \bullet }{   \bullet \circ } }(u)  
 +  B_{=; \genfrac{}{}{0pt}{}{\circ  \circ }{   \bullet \bullet }  }(u)
+  D B_{\neq; \genfrac{}{}{0pt}{}{\circ  \circ }{   \bullet \bullet }  }(u) 
+  D B_{\neq;\genfrac{}{}{0pt}{}{\circ  \bullet }{   \bullet \circ } }(u)  \crcr
& \qquad = \frac{  D(D-1) u^3 (1+Du ) + D (D-1)u^2(1+Du)}{(1-u^2)(1-D^2u^2)}  
=\frac{  D(D-1) u^2 }{(1-u )(1-D u )}  \; .
\end{align*}
 
 \end{itemize}
  The total contribution of these reduced schemes is:
  \begin{align*}
  & T(z)  \frac{1}{n } \binom{2n-2}{n-1}  \times \crcr
  & \times \left[ \left(   \frac{  D(D-1) u^2 }{(1-u )(1-D u )}  \right)^{2n-1} \left(  1 + \binom{D}{2} u   \right)^{n-1} 
  \left(   \frac{D(D+1)}{2} \frac{u^2}{1-u^2}   \right)^{n}  \right]_{u=zT(z)^{D+1}} \;,
  \end{align*}
  where the Catalan numbers count the choices of binary trees. Observe that for $n=1$  we recover the lollipop, but with the restriction that the vertical arm is a broken chain-vertex.
\item Case $b$, with $\delta=n\cdot D$: since $c^b_+=0$, there are no
  components of positive degree: the scheme is a graph with $3n-1$
  edges each carrying a broken chain-vertex, one node with degree one
  carrying the root ring and $2n-1$ nodes of degree 3 carrying
  either rings or $(D-2)$-dipoles. In particular this graph can be decomposed (in many
  ways) into a spanning tree with $2n-1$ edges and $n$ extra edges.
  The singular behavior of $G_{\tilde S}(u)$ for such a scheme $S$ is of the form: 
  \[
  \frac{1}{(1-Du)^{3n-1}}, \qquad(\delta=n\cdot D) \; .
  \]
 As the number of trivalent graphs grows super exponentially the
 corresponding series of dominant contributions is not summable in
 $n$.
\end{itemize}
In conclusion:
\begin{itemize}
\item For $D=3$, we get ${\bf B}^a=2\delta-1$ and ${\bf B}^b=\delta-1$ so $c^a_+ =
  \delta/(D-2), q_a = 0 $ gives the dominant contribution: for all
  $\delta>0$, the dominant schemes are the binary trees of Case $a$
  above.
\item For $D=4$, we get ${\bf B}^a=\delta -1 $ and ${\bf B}^b=3\delta/4-1$, so again
  $c^a_+ =\delta/(D-2)$ (hopefully here $\delta$ is always even)
  $q_a=0$ gives the dominant contributions and Case $a$ schemes
  dominate again, for all (even) $\delta>0$.
\item For $D=5$, we get ${\bf B}^a=2\delta/3-1$ and ${\bf B}^b=3\delta/5-1$, and
  binary trees again, but only for values of $\delta$ that are
  multiples of 3. For the other values of $\delta$, the dominant
  graphs have a lower ratio ${\bf B}/\delta$ (so that they do not appear in
  the scaling limit of next section).
\item For $D=6$, we get ${\bf B}^a=\delta/2-1$ and ${\bf B}^b=\delta/2-1$ so its a draw: all
  combinations are possible. The dominant graphs are both binary trees
  with loop leaves and 3-regular graphs, as well as mixed graphs.
\item For $D\geq 7$, we get ${\bf B}^a=2\delta/(D-2)-1<{\bf B}^b=3\delta/D-1$ so
  $q_b=\delta/D$ wins, and the dominant graphs are the rooted trivalent graphs.
\end{itemize} 

\subsection{Double scaling}

The Feynman amplitude of graphs (maps) in matrix models is $N^{2-2g}$,
where $g$ is the genus of the map.  In colored tensor models
\cite{1overN} the $(D+1)$-colored graphs come equipped with a
scaling in $N$, $N^{D - \delta}$ hence the leading behavior of the 
generating function of graphs with degree $\delta$ and ${\bf B}$
broken chains is $N^{D} \left[ N^{-\delta}  (z_0-z)^{ - \frac{ {\bf B} }{2} } \right]$.
In this context a natural question is whether the dominant terms of such a family of generating functions can be resummed. 

In the case $3\leq D\leq5$ the generating functions of colored graphs having
dominant reduced schemes can indeed be resummed to:
\begin{align*}
& T(z) \Bigg[ 1 + \sum_{n\ge 1}  \frac{1}{N^{(D-2)n}}
\frac{1}{n} \binom{2n-2}{n-1} \times    \crcr
& \qquad \qquad \times \left(   \frac{  D(D-1) u^2 }{(1-u )(1-D u )}  \right)^{2n-1} \left(  1 + \binom{D}{2} u   \right)^{n-1} 
  \left(   \frac{D(D+1)}{2} \frac{u^2}{1-u^2}   \right)^{n} \Bigg]\crcr
&=T(z) \Bigg[ 1  + \frac{1}{N^{D-2}}     \frac{  D(D-1) u^2 }{(1-u )(1-D u )}   \frac{D(D+1)}{2} \frac{u^2}{1-u^2}     \sum_{n\ge 0} \frac{1}{n+1} \binom{2n}{n} \times   \crcr
& \qquad \qquad \times \left[ \frac{1}{N^{D-2}} 
\left(    \frac{  D(D-1) u^2 }{(1-u )(1-D u )}  \right)^{2} \left(  1 + \binom{D}{2} u   \right)  \left(   \frac{D(D+1)}{2} \frac{u^2}{1-u^2}   \right) \right]^n \Bigg] \; ,
\end{align*}
and all the other generating functions of reduced schemes are either more suppressed in $1/N$ or less singular close to criticality.

Letting $N\to\infty$ and $z\to z_0$ while keeping $N^{D-2}(1-z/z_0) = x^{-1}$ fixed and large enough, the above sum over $n$ converges. 
For each such choice of $x$ we can define a non trivial limit distribution on the set of all rooted colored graphs with dominant reduced
schemes such that large schemes are favored in this distribution. 
The sum over dominant reduced schemes is:
\begin{align*}
 & T(z) \left\{1 +      
 \frac{
  1-\sqrt{1 - 4 \left[ \frac{1}{N^{D-2}} 
\left(    \frac{  D(D-1) u^2 }{(1-u )(1-D u )}  \right)^{2} \left(  1 + \binom{D}{2} u   \right)  \left(   \frac{D(D+1)}{2} \frac{u^2}{1-u^2}   \right) \right]}
 }{  2 \left[  
\left(    \frac{  D(D-1) u^2 }{(1-u )(1-D u )}  \right) \left(  1 + \binom{D}{2} u   \right)   \right]  }
 \right\} \crcr
& = T(z) \Bigg\{ 
 1 + \frac{ (1-u) (1-Du)}{ 2 D(D-1) u^2 \left( 1 + \binom{D}{2} u \right) } - \crcr
 & \qquad \qquad - \sqrt{  \frac{  (1-u)^2 (1-Du)^2 }{  4 D^2 (D-1)^2 u^4 \left(  1 + \binom{D}{2} u   \right)^2    }    
  - \frac{      \left(   \frac{D(D+1)}{2} \frac{u^2}{1-u^2}   \right)  }{N^{D-2}  \left(  1 + \binom{D}{2} u   \right)  }
 } 
\Bigg\}  \; .
\end{align*} 
Taking into account that for $z\to z_0$ we have:
\begin{align*}
&  T(z) \sim \frac{1+D}{D} - \sqrt{ 2 \frac{ D+1 }{D^3} } \sqrt{1- \frac{z}{z_0}} \;,  \quad u \sim \frac{1}{D}  - \sqrt{ 2 \frac{ D+1 }{D^3} } \sqrt{1- \frac{z}{z_0}} \crcr
& 1-Du \sim D \sqrt{ 2 \frac{ D+1 }{D^3 } } \sqrt{1- \frac{z}{z_0}} \;, 
\end{align*}
this becomes:
\[
 \frac{D+1}{D}  -   \sqrt{ 2  \frac{  D+1 }{D^3 }     }
 \sqrt{  1- \frac{z}{z_0}   -\frac{ D^2   }{N^{D-2} 2 (D-1)}
 } 
  +   \frac{2}{D^2}\sqrt{1- \frac{z}{z_0}}  \sqrt{  
  1- \frac{z}{z_0}   -\frac{1}{N^{D-2} } \frac{D^2}{2 ( D -1 )} 
 } 
 \; .
\]

Note that, as expected, we recover the singular behavior of $T(z)$ in the $N\to \infty$ limit. At finite $N$ the singularity is shifted
from $z_0$ to $z_1 = z_0 \bigl( 1-  \frac{D^2  } {N^{D-2}2(D-1)}   \bigr) <z_0$.
The new singularity governs the critical behavior of the double scaling series.

The double scaling regime identified in this paper must be studied
further.  The next step is to study the geometry of the resulting
continuum random space starting with its Hausdorff and spectral
dimensions. In the case $3\le D\le 5$, as the doubles scaling is summable, one can
attempt to iterate this procedure and construct a multiple scaling
limit in which a genuinely new random space is obtained.


\section*{Acknowledgements} 
GS thanks Vincent Rivasseau for bringing his attention to the
activities around melons in quantum gravity. Stephane Dartois and
Vincent Rivasseau are also thanked for interesting discussions on the
topic of this paper and the related \cite{ongoing}. The authors would like to thank two anonymous referees 
whose many detailed comments led to significant improvements of the manuscript.


\bibliographystyle{plain} 


\end{document}